\documentclass{amsart}
\usepackage{amssymb,amsrefs}
\usepackage[colorlinks]{hyperref}
\usepackage{booktabs}
\usepackage[all]{xy}

\theoremstyle{plain}
\newtheorem{theorem}{Theorem}[section]
\newtheorem{proposition}[theorem]{Proposition}
\newtheorem{lemma}[theorem]{Lemma}
\newtheorem{corollary}[theorem]{Corollary}
\newtheorem{conjecture}[theorem]{Conjecture}
\newtheorem{definition}[theorem]{Definition}

\theoremstyle{definition}
\newtheorem{remark}[theorem]{Remark}
\newtheorem{example}[theorem]{Example}

\DeclareMathOperator{\Art}{Art}
\DeclareMathOperator{\Cl}{Cl}
\DeclareMathOperator{\id}{id}
\DeclareMathOperator{\Gal}{Gal}
\DeclareMathOperator{\Hom}{Hom}
\DeclareMathOperator{\Mat}{Mat}
\DeclareMathOperator{\rank}{rank}
\DeclareMathOperator{\im}{im}
\DeclareMathOperator{\ord}{ord}
\DeclareMathOperator{\F2}{{{\FF}_2}}

\newcommand{\legendre}[2]{(\frac{#1}{#2})}
\newcommand{\Legendre}[2]{\left(\frac{#1}{#2}\right)}

\newcommand{\lqedhere}{\ensuremath{\text{\qedhere\hspace{-1em}}}}

\newcommand{\FF}{{\mathbf{F}}}
\newcommand{\RR}{{\mathbf{R}}}
\newcommand{\Que}{{\mathbf{Q}}}
\newcommand{\Zee}{{\mathbf{Z}}}

\newcommand{\gotha}{{\mathfrak{a}}}
\newcommand{\gothc}{{\mathfrak{c}}}
\newcommand{\gothm}{{\mathfrak{m}}}
\newcommand{\gothp}{{\mathfrak{p}}}

\newcommand{\Ocal}{{\mathcal{O}}}
\newcommand{\calD}{{\mathcal{D}}}
\newcommand{\calF}{{\mathcal{F}}}

\newcommand{\tto}{\longrightarrow}
\newcommand{\equi}{\Longleftrightarrow}
\newcommand{\congr}{\equiv}
\newcommand{\iso}{\cong}
\newcommand{\mapright}[1]{\mathop{\longrightarrow}\limits^{#1}}
\newcommand{\ab}{\textup{ab}}

\renewcommand{\mod}{\bmod}

\newcommand{\mybar}[1]{
  \mathchoice
  {#1\llap{$\overline{\phantom{\displaystyle\rm#1}}$}}
  {#1\llap{$\overline{\phantom{\textstyle\rm#1}}$}}
  {#1\llap{$\overline{\phantom{\scriptstyle\rm#1}}$}}
  {#1\llap{$\overline{\phantom{\scriptscriptstyle\rm#1}}$}}
}  
\renewcommand{\bar}{\mybar}

\begin{document}

\title[Redei reciprocity, governing fields, and negative Pell]
{Redei reciprocity, governing fields, and negative Pell}

\author[Peter Stevenhagen]{Peter Stevenhagen}
\address{Mathematisch Instituut,
         Universiteit Leiden, 
         Postbus 9512, 
         2300 RA Leiden, The Netherlands}
\email{psh@math.leidenuniv.nl}

\date{\today}


\begin{abstract}
We discuss the origin, an improved definition and the key reciprocity property
of the trilinear symbol introduced by R\'edei \cite{Redei}
in the study of 8-ranks of narrow class groups of quadratic number fields.
It can be used to show that such 8-ranks are `governed' by Frobenius conditions
on the primes dividing the discriminant, a fact used
in the recent work of A. Smith \cites{Smith, Smith2}.
In addition, we explain its impact in the progress towards
proving my conjectural density
for solvability of the negative Pell equation $x^2-dy^2=-1$.
\end{abstract}
\maketitle

\section{Introduction}
\label{S:intro}

\noindent
In a 1939 Crelle paper \cite{Redei}, the Hungarian
mathematician L\'aszl\'o R\'edei introduced a trilinear quadratic symbol
$[a,b,c]\in\{\pm1\}$ for quadratic discriminants $a, b\in\Zee$
and positive squarefree integers $c$ satisfying a number of conditions.
He used his symbol to describe 8-ranks of quadratic class groups,
much in the way he had described the 4-ranks of these class groups
in terms of Legendre symbols in his earlier work~\cite{Redei4}.
His definition of the symbol, as a Jacobi symbol in the quadratic field
$\Que(\sqrt a)$, is somewhat involved, and seems to depend on many choices.
Moreover, it only allows for a limited `symmetry' of the symbol in
its arguments, as infinite primes are disregarded in his definition.

An improved definition in class field theoretic terms was 
proposed in 2007 by Jens Corsman \cite{Corsman}.
He imposes fewer conditions on the arguments of the symbol,
which are most conveniently taken in the group $\Que^*/{\Que^*}^2$
of non-zero rational numbers modulo squares,
requiring them to have relative quadratic Hilbert symbols
\begin{equation}
\label{eqn:hilbabc}
(a,b)_p=(a,c)_p=(b,c)_p=1
\end{equation}
at all primes $p$, and to satisfy the coprimality condition 
\begin{equation}
\label{eqn:abccoprime}
\gcd(\Delta(a),\Delta(b),\Delta(c))=1
\end{equation}
for the discriminants of the associated quadratic fields, with 
$\Delta(1)=1$.
The R\'edei symbol is then defined as a product
$[a,b,c]=\prod_{p|c} [a,b,c]_p$
of `local' symbols at the primes
$p$ dividing the squarefree integer representing $c$.
It is essential to include the infinite
prime $p=\infty$ in the product, with $\infty|c$ 
having the meaning $c<0$.
A correct definition of $[a,b,c]_p$
leads to a striking feature that we baptize
\emph{R\'edei's reciprocity~law\/}.
\begin{theorem}
\label{thm:reclaw}
For $a, b, c \in\Que^*/{\Que^*}^2$ satisfying
\eqref{eqn:hilbabc} and \eqref{eqn:abccoprime},
the R\'edei symbol $[a,b,c]$ 
is linear in each of its arguments, and satisfies
the reciprocity~law
$$
[a,b,c]=[b,a,c]=[a,c,b].
$$
\end{theorem}

\noindent
The R\'edei symbol traditionally has its values in $\{\pm 1\}$,
as it may be computed as a Jacobi symbol in a quadratic field.
Moreover, it has a definition as a product of local symbols that
in \eqref{eqn:ppartishilb} turn out to be
quadratic Hilbert symbols, which invariably have
values in $\{\pm 1\}$, and satisfy a \emph{product} formula
on which the proof of Theorem~\ref{thm:reclaw} is based.
However, the R\'edei symbol is typically used in a linear algebra setting
over the field of two elements $\FF_2$.
It therefore makes sense to take its value in $\FF_2$, as we will do
in our formal Definition \ref{def:redeisymbol}.

Part of the perfect symmetry of the symbol $[a,b,c]$ in its arguments
is immediate from the definition,
as $[a,b,c]$ is an Artin symbol depending on $c$
in a cyclic quartic extension $K=\Que(\sqrt{ab})\subset F_{a,b}$
that depends symmetrically on $a$ and~$b$.
Symmetry involving $c$ is a true reciprocity:
we may swap $b$ and $c$ in the symbol by a (non-obvious)
application of quadratic reciprocity over $\Que(\sqrt a)$.

The auxiliary field $F_{a,b}$ occurring in the definition of $[a,b,c]$,
which is only unique up to twisting by a finite group $T_{a,b}$
of quadratic characters \eqref{eqn:Tab},
is the most complicated ingredient in the definition. 
Corsman failed to notice that $K\subset F_{a,b}$ may be ramified over 2,
and that one has to require \emph{minimal} ramification at 2
for $K\subset F_{a,b}$ in order for $[a,b,c]$ to be well-defined, 
independently of the choice of $F_{a,b}$.

In the case of prime arguments $a,b,c \congr 1\mod 4$,
no dyadic ramification subtleties arise, and the symbol has been
interpreted by Morishita \cite{Morishita}*{Section 8.2}
as an arithmetic Milnor invariant, leading to a description
as a triple Massey product that is useful in the study of
pro-2-extensions of $\Que$ with given ramification locus~\cite{Gaertner}.

Although Galois cohomology does play a role in Corsman's approach
to the R\'edei symbol, its relation to Massey products
and the applications of R\'edei reciprocity
to the \emph{average} behavior of the 2-part of 
imaginary quadratic class groups
in the recent work of Smith \cites{Smith, Smith2},
neither the definition of the symbol nor the proof of its symmetry
properties needs it, and we do not use it in this paper.
Galois cohomology may be needed to find generalizations of the R\'edei symbol.
In fact, the linearity properties of the symbol make it one 
of the rare trilinear maps that `naturally occur' in mathematics,
and it is an interesting question whether the symbol
has variants having properties of cryptographic interest
in the sense of \cite{BonehSil}.

\medskip\noindent
In this paper we approach R\'edei's symbol along historical lines, showing
how it arises in the study of the 2-part of the
narrow class group $C$ of a quadratic field~$K$ of discriminant~$D$.
Starting from old results in Section \ref{section:2-rank}
on the 2-rank of~$C$, we describe the 4-rank of $C$ in terms of
the \emph{R\'edei matrix} $R_4=R_4(D)$, a matrix over the field $\FF_2$
of 2 elements with entries that are essentially $\FF_2$-valued
relative Legendre symbols of the primes dividing~$D$
(Theorem \ref{theorem:4-rank}).
Linear algebra also gives the 8-rank of $C$ in terms of
a matrix $R_8=R_8(D)$ over $\FF_2$ (Theorem \ref{theorem:8-rank}),
but this time its entries are ($\FF_2$-valued)
\emph{R\'edei symbols} $[d_1, d_2, m]$,
given in Definition~\ref{def:redeisym} as the Artin symbol of
an ambiguous ideal in $K$ of norm $m$
in an unramified cyclic quartic extension 
$K=\Que(\sqrt{d_1d_2})\subset F_{d_1, d_2}$
having $\Que(\sqrt {d_1}, \sqrt{d_2})$ as its intermediate
quadratic extension.
In Section \ref{S:Computing Redei-symbols},
we explicitly compute $[d_1,d_2,m]$ in a field $F_{d_1, d_2}$
obtained by twisting the field $F(x,y,z)$ in \eqref{eqn:Fxyz}
associated to a primitive integral point on the conic 
$$
x^2-d_1y^2-d_2z^2=0
$$
by a quadratic character to ensure that $K\subset F_{d_1, d_2}$
is unramified at 2.

Section \ref{S:Discovering Redei reciprocity}
shows how R\'edei's reciprocity law is suggested by
the behavior of small examples, and indicates how the 
general symbol $[a, b,c]$ should be defined in order to
obtain reciprocity.
The precise definitions are in Section \ref{S:Redei symbols},
leading to a proof of Theorem \ref{thm:reclaw} in
Section \ref{S:Proving Redei reciprocity}.

As an immediate application of R\'edei reciprocity,
Section~\ref{S:Governing fields} shows, following Corsman,
how it yields the existence of
\emph{governing fields} for the 8-rank of class groups
in 1-parameter families $\Que(\sqrt{dp})$, with $d$ a fixed 
integer and $p$ a variable prime, a result that was originally
obtained by different means in 1988 in \cite{Stev}, and that
is at the basis of Smith's work.
Section 10 discusses its impact on my now classical conjecture
\ref{conj:negpell} on the number of real quadratic fields with fundamental
unit of norm $-1$ or, equivalently, the asymptotic number of squarefree
$d\in\Zee_{>1}$ for which the \emph{negative Pell equation}
\begin{equation}
x^2-dy^2=-1
\end{equation}
is solvable in integers $x, y\in\Zee$.

\section{The 2-rank}
\label{section:2-rank}

\noindent
Let $d\ne 1$ be a squarefree integer, $K=\Que(\sqrt d)$ the
corresponding quadratic field, $D\in\{d, 4d\}$ the discriminant of $K$,
and $C=\Cl^+_K=\Cl^+(\Ocal_K)$ the {\it narrow\/} class group of $K$, i.e.,
the quotient $C=I/P^+$ of the group $I$ of fractional $\Ocal_K$-ideals
by the subgroup of principal ideals $(x)=x\Ocal_K$ with
generator of \emph{positive} norm $N(x)$.
The narrow class group maps surjectively to 
the ordinary class group $\Cl_K$ of $K$, and
we have an exact sequence
\begin{equation}
\label{eqn:narrow-ordinary}
0\to \langle F_\infty\rangle \tto C\tto \Cl_K\to 0
\end{equation}
in which $F_\infty$, the \emph{Frobenius at $\infty$}, denotes
the ideal class $[(\sqrt d)]\in C$.
This is the trivial element in $C$ if $K$ is imaginary quadratic, and 
also if $K$ is real quadratic with fundamental unit $\varepsilon_d$
of norm $N(\varepsilon_d)=-1$.
If $K$ is real quadratic with $N(\varepsilon_d)=1$,
then $F_\infty$ is of order 2, and $C$ has twice the size of $\Cl_K$.

Describing the 2-part $C[2^\infty]\subset C$ consisting of all 
2-power torsion elements in~$C$
can be done by specifying, for $k\ge1$, the $2^k$-rank
$$
r_{2^k}=r_{2^k}(D)= \dim_{\F2} C[2^k]/C[2^{k-1}]=\dim_{\F2} 2^{k-1}C/2^kC.
$$
The sequence of non-negative integers $r_2, r_4, r_8, \ldots$
is non-increasing, and we have $r_{2^k}=0$ for $k$ sufficiently large.

The 2-rank $r_2=\dim_{\F2} C[2]=\dim_{\F2} C/2C$
was already determined by Gauss, 
who defined~$C$ in terms of binary quadratic forms.
To state his result, we factor $D$ as a product
\begin{equation}
\label{eqn:Dfactorization}
D=\prod_{i=1}^t p^*_i = t_D \prod_{p|D \text{ odd}} p^*
\end{equation}
of signed odd prime discriminants $p^*=(-1)^{(p-1)/2}p\congr 1\mod 4$
and a discriminantal 2-part $t_D\in\{1, -4, \pm 8\}$
that we sloppily denote by $2^*$ in case $D$ is even.
We let $\gothp_i|p_i$ be the prime of $K$ lying over $p_i$.
It satisfies $\gothp_i^2=(p_i)$, so we have $[\gothp_i]\in C[2]$.
\begin{theorem}
\label{theorem:2-rank}
We have $r_2=t-1$, with $t$ the number of prime divisors of~$D$.
\end{theorem}

\begin{proof}
There are two fundamentally different proofs of this result, describing
$C[2]$ and $C/2C$, respectively.
The first uses the $t$ \emph{ambiguous ideal classes}
$[\gothp_i]\in C[2]$ coming from the ramifying primes
$\gothp_i|p_i$ of $K$, the second the $t$ \emph{genus characters}
$\chi_{p_i^*}\in\widehat C[2]$ corresponding to the
discriminantal divisors $p^*_i$ in \eqref{eqn:Dfactorization}.

In the first proof, one exploits the Galois action on $C$ of
$\Gal(K/\Que)=\langle\sigma\rangle$, noting that
$\sigma$ acts by inversion as the norm map $N=1+\sigma$ annihilates $C$.
A little Galois cohomology shows that the 2-torsion subgroup $C[2]=C[\sigma-1]$
is generated by the $t$ classes $[\gothp_i]$, subject to a single relation.
This yields $r_2=t-1$.

For the second proof, one views $C=\Gal(H/K)$ under the Artin isomorphism
as the Galois group over $K$ of the narrow Hilbert class field $H$ of $K$.
Then $H$ is Galois over $\Que$ with dihedral Galois group
$$
\Gal(H/\Que)\iso \Gal(H/K)\rtimes \Gal(K/\Que)=C\rtimes \langle\sigma\rangle,
$$
as the surjection $\Gal(H/\Que)\to \Gal(K/\Que)=\langle\sigma\rangle$
is split and $\sigma$ acts by inversion.
The \emph{genus field} $H_2\subset H$ of $K$, which is defined as the
maximal subfield of $H$ that is abelian over $\Que$, has 
as its Galois group over $\Que$ the elementary abelian 2-group
\begin{equation}
\label{eqn:galh2q}
\Gal(H_2/\Que)=\Gal(H/\Que)^\ab = C/2C \times \langle\sigma\rangle.
\end{equation}
One can generate $H_2$ explicitly over $\Que$ by $t$ independent
square roots as
\begin{equation}
\label{eqn:genusfield}
H_2=\Que(\{\sqrt {p_i^*} : i= 1,2, \ldots, t\}),
\end{equation}
so $\Gal(H_2/\Que)$ is naturally an $\F2$-vector space of dimension $t$,
in which the subspace $C/2C=\Gal(H_2/K)\subset \Gal(H_2/\Que)$
has dimension $r_2=t-1$.
\end{proof}

\smallskip\noindent
The second proof of Theorem \ref{theorem:2-rank} shows that
the prime power discriminants $p_i^*|D$ in \eqref{eqn:Dfactorization}
yield an $\FF_2$-basis of the quadratic characters
on $\Gal(H_2/\Que)$, with
\begin{equation}
\label{eqn:chipi}
\chi_{p_i^*}: \Gal(H_2/\Que) \to \Gal(\Que(\sqrt {p_i^*})/\Que)\iso \F2
\end{equation}
giving the Galois action on $\sqrt {p_i^*}$.
Even though this Galois action is given by multiplication by $\pm1$,
it is convenient for our linear algebra purposes to
have \emph{additive} characters with values in $\Que/\Zee$,
and define the quadratic characters $\chi_{p_i^*}$ with values in
$\frac{1}{2}\Zee/\Zee=\FF_2$.

The character $\chi_{d_1}=\sum_{i\in S} \chi_{p_i}$
for a subset $S\subset \{1,2,\ldots, t\}$ corresponding to the
\emph{discriminantal divisor} $d_1=\prod_{i\in S} p_i^*$ of~$D$ gives the 
action on $\sqrt {d_1}$.
When restricted to $C/2C=\Gal(H_2/ K)\subset \Gal(H_2/\Que)$,
it yields a quadratic character in the character group 
\begin{equation}
\label{eqn:c-hat}
\widehat C = \Hom (C,\Que/\Zee)
\end{equation}
of $C$ that coincides with the character $\chi_{d_2}$
corresponding to the complementary divisor
$d_2=D/d_1=\prod_{i\notin S} p_i^*$.
R\'edei calls an unordered pair $(d_1, d_2)$ of quadratic discriminants
satisfying
\begin{equation}
\label{eqn:discdec}
D=d_1d_2
\end{equation}
a \emph{discriminantal decomposition} of~$D$.
It `is' the \emph{genus character}
$\chi_{d_1}=\chi_{d_2}\in \widehat C[2]$, and the
corresponding finitely unramified quadratic extension $K\subset E$
inside~$H$ is
\begin{equation}
\label{eqn:F2}
E=K(\sqrt {d_1})=\Que(\sqrt {d_1}, \sqrt {d_2})=K(\sqrt {d_2}).
\end{equation}

\section{The 4-rank and negative Pell}
\label{section:4-rank}

\noindent
The ambiguous ideal proof of Theorem \ref{theorem:2-rank} describes
the subgroup $C[2]\subset C$ as a quotient
of $\FF_2^t$ by a surjection
\begin{equation}
\label{eqn:alpha}
\alpha: \FF_2^t \tto C[2]
\end{equation}
that sends the $j$-th basis vector to the class $[\gothp_j]$.
The generator $A_D\in \FF_2^t$ of its 1-dimensional kernel 
encodes the unique non-trivial relation that exists
between the $t$ classes of the ramifying primes $\gothp_j|D$ of~$K$. 
In particular, it tells us whether the element
$F_\infty=[(\sqrt d)]\in C[2]$ in \eqref{eqn:narrow-ordinary} is trivial.
In the interesting case $D>0$, this amounts to
the fundamental unit of $K=\Que(\sqrt D)$ having norm $N(\varepsilon_d)=-1$.

The genus theory proof of Theorem \ref{theorem:2-rank} describes
the quotient $C/2C=\Gal(H_2/K)$ of~$C$ as a subspace
of $\Gal(H/\Que)=\FF_2^t$ under the inclusion map
\begin{equation}
\label{eqn:gamma}
\gamma: C/2C=\Gal(H_2/K)\tto \Gal(H_2/\Que)=\FF_2^t,
\end{equation}
with the $i$-th coordinate of $\gamma(\gotha)\in \FF_2^t$ for $[\gotha]\in C$
describing the action of the Artin symbol $\Art(\gotha, H/K) \in\Gal(H/K)$
on $\sqrt{p_i^*}$.
As elements of $\Gal(H_2/K)$
fix the product $\prod_{i=1}^t\sqrt{p_i^*}=\sqrt D$,
the map $\gamma$ embeds $C/2C$ as the `sum-zero-hyperplane' in $\FF_2^t$.
Equivalently, one can formulate this as in \eqref{eqn:alpha}
by saying that the subgroup
$\widehat C[2]\subset \widehat C$ of quadratic characters on $C$
is generated by the $t$ characters $\chi_{p_i^*}$, subject to the relation
that their sum
\begin{equation}
\label{eqn:chiD}
\chi_D= \sum_{i=1}^t \chi_{p_i^*},
\end{equation}
the Dirichlet character corresponding to $K$, 
is the character on $\Gal(H_2/\Que)$ in \eqref{eqn:galh2q}
that has kernel $C/2C$ and is trivial as an element of $\widehat C$.
This time, the relation holds no deeper information as it
is `the same' for all quadratic fields.

The 4-rank of $C$ is the $\F2$-dimension of the kernel
$C[2]\cap 2C$ of the natural~map 
$$
\varphi_4: C[2]\to C/2C,
$$ 
and we can find it by combining $\varphi_4$ with the surjection
$\alpha$ and the injection $\gamma$
from \eqref{eqn:alpha} and \eqref{eqn:gamma}
into a single $\F2$-linear \emph{R\'edei map}
\begin{equation}
\label{eqn:R4}
R_4:\quad
\FF_2^t \mapright{\alpha} C[2]
	\mapright{\varphi_4} C/2C \mapright{\gamma} \Gal(H_2/\Que)=\FF_2^t.
\end{equation}
We have
$$
1+r_4=
1+\dim_{\F2} \ker \varphi_4=\dim_{\F2} \ker  R_4 = t-\rank_{\F2} R_4,
$$
and writing $r_2=t-1$ as in Theorem \ref{theorem:4-rank},
we obtain the following result.
\begin{theorem}
\label{theorem:4-rank}
The $4$-rank of $C$ equals
$r_4 = r_2 - \rank_{\F2} R_4$. 
\qed
\end{theorem}

\noindent
Explicit entries for the matrix
$R_4=(\varepsilon_{ij})_{i, j}\in\Mat_{t \times t}(\F2)$ 
can easily be given.
The entry $\varepsilon_{ij}$ describes
the action of the Artin symbol $\Art (\gothp_j, H/K)$
on $\sqrt{p_i^*}\in H_2\subset H$.
For $i\ne j$, it is an $\FF_2$-valued Legendre (or for $p_j=2$ Kronecker)
symbol:
\begin{equation}
\label{eqn:R4entries}
\varepsilon_{ij}=\chi_{p_i^*}([\gothp_j])=
\Legendre{p_i^*}{p_j}\in\FF_2.
\end{equation}
The diagonal entries $\varepsilon_{jj}=\sum_{i\ne j} \varepsilon_{ij}$
of $R_4$ follow from the sum-zero-property of $\gamma$:
the rows of $R_4$ add up to $0\in\FF_2^t$.
This simple description of $r_4$ in terms of the relative 
quadratic behavior of the primes $p_i$ dividing $D$
goes back to R\'edei \cite{Redei4}.

The somewhat hybrid notation in \eqref{eqn:R4entries} uses
an identification `$\{\pm1\}=\FF_2$' of multiplicative and additive 
value groups of quadratic symbols and characters.
The same notational ambiguity inevitably occurs for R\'edei symbols,
which are quadratic symbols in quadratic fields, with
values that are traditionally taken in $\{\pm1\}$, but that also
occur as entries of matrices over $\FF_2$.
We have mostly chosen additive values of characters and symbols in this 
paper, but multiplicative values are used in the proof of 
R\'edei's reciprocity law in Section \ref{S:Proving Redei reciprocity},
which relates R\'edei symbols to quadratic Hilbert symbols.

\bigskip

\noindent
For $K=\Que(\sqrt d)$ real quadratic of discriminant $D\in\{d, 4d\}$,
the fundamental unit has norm $N(\varepsilon_d)=-1$ if and only if the
\emph{negative Pell equation}
\begin{equation}
\label{eqn:negpell}
x^2-dy^2=-1
\end{equation}
is solvable in \emph{integers} $x, y\in\Zee$.
If it exists, the smallest solution to \eqref{eqn:negpell}
can be found from the continued fraction expansion of $\sqrt d$,
which then needs to have \emph{odd} period length,
or from general unit finding algorithms in number rings \cite{Lenstra}.
For $d=D\congr 5\mod 8$,
this solution corresponds to the \emph{cube} of $\varepsilon_d$ in case
the fundamental unit
$\varepsilon_d\in\Ocal_K=\Zee[(1+\sqrt d)/2]$
does not lie in $\Zee[\sqrt d]$.
Conjecturally~\cite{StevEis}, this happens 
for a fraction $2/3$ of squarefree $D\congr 5\mod 8$.

For solvability of \eqref{eqn:negpell} in \emph{rational numbers},
which amounts to $K$ having elements of norm $-1$ or, equivalently,
having quadratic Hilbert symbols
\begin{equation}
\label{eqn:-1norm}
(d,-1)_p=(D, -1)_p=1\qquad\hbox{for all primes $p\le\infty$,}
\end{equation}
there is an easy criterion:
solvability occurs if and only if $d$ (or $D$) is positive and without prime
factors $p\congr 3\mod 4$.
By \cite{Rieger}*{Satz 3}, the set $\calD$ of such $D$ is a thin set,
asymptotically containing $cX/\sqrt{\log X}$ elements $D<X$,
for some explicit $c\approx .348\in \RR_{>0}$.
For $D\in \calD$ we have $t_D\in\{1, 8\}$ and $p^*=p$
in \eqref{eqn:Dfactorization}, and by (7)
\begin{equation}
\label{eqn:H2real}
D \in \calD
\quad\Longleftrightarrow\quad
H_2 \hbox{ is totally real.}
\end{equation}
The class field theoretic implication of \eqref{eqn:narrow-ordinary}
is that the set ${\calD}^-\subset \calD$ of discriminants for which 
the negative Pell equation \eqref{eqn:negpell} is solvable in integers
has a similar description:
\begin{equation}
\label{eqn:Hreal}
D \in \calD^-
\quad\Longleftrightarrow\quad
H \hbox{ is totally real.}
\end{equation}
Indeed, the narrow Hilbert class field $H$ of $K=\Que(\sqrt d)$
is totally real if and only if the
Frobenius at infinity $F_\infty\in C$ at the real
primes of $K=\Que(\sqrt d)$ is trivial on $H$.

Assume $D\in \calD$.
Then the map
$\alpha: \FF_2^t\to C[2]$ in \eqref{eqn:alpha} describes
$F_\infty=[(\sqrt d)]$ as
$$
F_\infty=\alpha[(1)_{i=1}^t] \in C[2]\cap 2C,
$$
and $V=\alpha^{-1}(C[2]\cap 2C)\subset \FF_2^t$
is an $(r_4+1)$-dimensional subspace containing~$(1)_{i=1}^t$.
The linear map $\alpha|_V: V\to C[2]\cap 2C$ is surjective with
kernel $\FF_2\cdot A_D$, and
numerical evidence \cite{BosSte} suggests that $A_D$ behaves like
a `random' non-zero element in $V$.
As $D\in\calD$ satisfies
$$
D\in {\calD}^-
\Longleftrightarrow
A_D=(1)_{i=1}^t,
$$
we expect that for the discriminants $D\in\calD$ having 4-rank $r_4=e$,
the negative Pell equation \eqref{eqn:negpell}
will be solvable with `probability' $(\#V-1)^{-1}=(2^{r_4+1}-1)^{-1}$.

To heuristically find the density for $\calD^-$ in $\calD$, 
write $\calD=\bigcup_{e=0}^\infty \calD(e)$, with $\calD(e)$ the subset
of $D\in \calD$ having 4-rank $r_4=e$. 
By Theorem \ref{theorem:4-rank}, a discriminant $D\in\calD$ is
in $\calD(e)$ if and only if its R\'edei matrix
$R_4\in \Mat_t(\FF_2)$ has corank $e+1$.
Quadratic reciprocity for the entries \eqref{eqn:R4entries} implies 
that the matrix $R_4$ is \emph{symmetric}, with rows and columns
adding up to 0.
The $(1,1)$-minor of $R_4$, which determines the
full matrix $R_4$, behaves as a random symmetric $(t-1)\times(t-1)$-matrix.
As the average number $t$ of primes factors of $D$ tends to infinity
with $D$, albeit very slowly, as $\log\log D$, we expect
the density of $\calD(e)$ in $\calD$ to equal
$P(e)=\lim_{n\to\infty}P_n(e)$, with $P_n(e)$ the fraction
of the ${n+1\choose 2}$
symmetric $n\times n$-matrices over $\FF_2$ having corank $e$.
In terms of the infinite product
\begin{equation}
\label{eqn:densalpha}
\alpha = \prod_{j\ \text{odd}} (1-2^{-j})=\prod_{j=1}^\infty (1+2^{-j})^{-1}
\approx .419422441,
\end{equation}
we have $P(0)=\alpha$, and more generally
$P(e)=\alpha\cdot \prod_{j=1}^e (2^j-1)^{-1}$ for $e\ge0$.

By the probability argument above,
we expect the natural density of $\calD^-(e)$ in $\calD(e)$ to be
$1/(2^{e+1}-1)$ for all $e\in\Zee_{\ge0}$. 
This leads to my 1992 conjecture for the solvability of the 
negative Pell equation \cite{StevPell}.
\begin{conjecture}
\label{conj:negpell}
The set $\calD^-$ of discriminants of quadratic fields
with fundamental unit of norm $-1$ has natural density
$$
\sum_{e=0}^\infty \frac{P(e)}{2^{e+1}-1} =
1-\alpha \approx .580577559
$$
inside the set $\calD$ of discriminants of quadratic fields
containing elements of norm~$-1$.
\end{conjecture}

\noindent
Fouvry and Kl\"uners proved in 2010 \cite{FouvKl} that
$\calD(e)$ does indeed have the expected density $P(e)$ in $\calD$.
For $e=0$ we have $2^{e+1}-1=1$ and $\calD(0)\subset \calD^-$,
as in the case $r_4=0$ the narrow Hilbert class field $H$ is totally real,
being a normal number field of \emph{odd} degree over the totally
real genus field $H_2$.
This immediately implies that $P(0)=\alpha$
is a lower bound for the lower density of $\calD^-$ in $\calD$.
To get better bounds, one needs control over the archimedean character of
$H$ for $e\ge1$.
We will address this in Section~\ref{S:negative Pell}.

\section{The 8-rank}
\label{section:8-rank}

\noindent
The 8-rank $r_8$ of $C$ equals the $\F2$-dimension of the kernel
of the natural map 
$$
\varphi_8: C[2]\cap 2C \tto 2C/4C
$$
between $r_4$-dimensional vector spaces over $\F2$.
Under the Artin isomorphism, the group $2C/4C$ is the Galois group
$\Gal(H_4/H_2)$, with $H_4\subset H$ the narrow 4-Hilbert class field of $K$.
We can restrict $\alpha$ in \eqref{eqn:alpha} to the kernel of
the 4-rank map $R_4$ from \eqref{eqn:R4} and compose with $\varphi_8$
to obtain an $\F2$-linear map
\begin{equation}
\label{eqn:R8}
R_8:\quad
\ker  R_4 \mapright{\alpha} C[2]\cap 2C
		\mapright{\varphi_8} 2C/4C=\Gal(H_4/H_2)\iso \FF_2^{r_4}
\end{equation}
defined on the $(r_4+1)$-dimensional space $\ker  R_4$.
Here we write $\Gal(H_4/H_2)\iso \FF_2^{r_4}$, in contrast to the equality
$\Gal(H_2/\Que)=\FF_2^t$ in \eqref{eqn:R4}, as we no longer have
an obvious choice for the basis of this $\FF_2$-vector space.
As $r_8 = \dim_{\F2} \ker \varphi_8$ is the codimension
of the image of $\varphi_8$ in $2C/4C$, we obtain the following
analogue of Theorem \ref{theorem:4-rank}.
\begin{theorem}
\label{theorem:8-rank}
The $8$-rank of $C$ equals
$r_8 = r_4-\rank_{\F2} R_8$.
\qed
\end{theorem}

\noindent
In order to obtain a matrix representing $R_8$,
we want to represent $\Gal(H_4/H_2)$ in \eqref{eqn:R8}
as explicitly as we represented $\Gal(H_2/\Que)$ in \eqref{eqn:R4}.
R\'edei had already achieved this in a 1934 paper
with Reichardt \cite{RedRei}, which computed $r_4$ not as in
Theorem \ref{theorem:4-rank} but by an
explicit construction of 
$H_2\subset H_4$ in terms of $r_4$ cyclic quartic extensions
$K\subset F$ inside~$H$ that are $K$-linearly disjoint.
For such unramified $F$, the group $\Gal(F/\Que)$ is dihedral of order~8,
and the intersection $E=F\cap H_2$ equals $E = \Que(\sqrt {d_1}, \sqrt {d_2})$
for a discriminantal decomposition $D=d_1d_2$ as in \eqref{eqn:discdec}
and~\eqref{eqn:F2}.

R\'edei calls the decompositions $D=d_1d_2$ defining those
$E=\Que(\sqrt {d_1}, \sqrt {d_2})$ that arise as $F\cap H_2$ \emph{zweiter Art},
`of the second kind'.
For the corresponding quadratic character
$\chi\in \widehat C =\Hom (C,\Que/\Zee)$, it means
that we have $\chi=2\psi$ for a quartic character
$\psi$ defining~$K\subset F$.
By the duality of finite abelian groups,
we have~$\chi\in 2\widehat C$ if and only if $\chi$ vanishes on
the $2$-torsion subgroup~$C[2]$.
This leads to the following characterization of these quadratic characters.
\begin{lemma}
\label{lemma:F2-criterion}
For a quadratic character
$\chi\in \widehat C[2]$ defining $E=\Que(\sqrt{d_1},\sqrt{d_2})$
as in \eqref{eqn:discdec} and \eqref{eqn:F2}, having 
$\chi\in 2\widehat C$ is equivalent to each
of the following:
\begin{enumerate}
\item
there exists a cyclic quartic extension $K\subset F$ inside $H$
containing $E$;
\item
$\chi_{d_1}\circ R_4=\chi_{d_2}\circ R_4$ is the zero map;
\item
all ramified primes of $K$ split completely in $K\subset E$;
\item
for $i=1,2$ and $p|d_i$ prime we have $\legendre{D/d_i}{p}=1$.
\end{enumerate}
\end{lemma}

\begin{proof}
Having $\chi=2\psi$ for a quartic character $\psi\in\widehat C$ 
defining $F$ as in (1) is equivalent to $\chi$ vanishing on
the subgroup $C[2]$ of ambiguous ideal classes generated by the classes of
the ramifying primes $\gothp|D=d_1d_2$ of $K$ as in \eqref{eqn:alpha}.
One can phrase this using the map $R_4$ from \eqref{eqn:R4} as in
(2), or in terms of the splitting of the ramifying primes in $K\subset E$
as in (3).
A ramifying prime of $K$ divides exactly one of $d_1$, $d_2$.
A prime $\gothp|p$ in $K$ dividing, say, $d_1$
splits completely in $K\subset E=K(\sqrt{d_2})$
if and only if the Legendre (or Kronecker) symbol 
$\legendre{d_2}{p}$ equals~1, as in (4).
\end{proof}
\begin{remark}
The identity $\chi=2\psi$ determines $\psi\in\widehat C$
up to a quadratic character, as an element of
$\widehat C[4]/\widehat C[2]$, and this means that the
quadratic extension $E\subset F$ it gives rise to in (1) 
is determined by $\chi$ only up to `twisting' by an
unramified quadratic character.
In other words: not $E\subset F$, but
the quadratic extension $H_2\subset H_2F$
it generates over the genus field $H_2$ is unique. 
\hfill$\lozenge$
\end{remark}

\noindent
In order to compute the 8-rank in Theorem \ref{theorem:8-rank}
from the rank of an explicit matrix describing the map $R_8$ in \eqref{eqn:R8},
we \emph{choose} an $\F2$-basis for the $(r_4+1)$-dimensional
subspace $\ker  R_4\subset \FF_2^t$,
and write 
$$
[\gothm_j]\in C[2]\cap 2C\qquad\quad (j=1,2,\ldots, r_4+1)
$$
for the images of these basis vectors under the map $\alpha$ from
\eqref{eqn:alpha}. 
The classes $[\gothm_j]$ span $C[2]\cap 2C$, subject to a single relation
encoded in $A_D$,
and their Artin symbols $\Art (\gothm_j, H/K)$ are the identity
on the genus field $H_2$.

Similarly, we \emph{choose} quartic characters $\psi_i$ for $i=1,2,\ldots, r_4$
spanning $\widehat C[4]/\widehat C[2]$, and let $K\subset F_{4,i}$
be the corresponding unramified quartic extensions.
By condition (2) of Lemma \ref{lemma:F2-criterion}, the
quadratic characters $\chi_i=2\psi_i\in\widehat C[2]$
come from vectors in $\FF_2^t$ that,
together with $(1)_{i=1}^t$, span the kernel 
of the \emph{transpose} $R_4^T$ of the R\'edei matrix in \eqref{eqn:R4}.
The $r_4$ quadratic extensions $H_2\subset H_2F_{4,i}$ span $H_2\subset H_4$,
and the map $R_8$ is represented by a matrix
$R_8=(\eta_{ij})_{i,j}\in \Mat_{r_4\times (r_4+1)}(\F2)$ with entries
\begin{equation}
\label{eqn:R8entries}
\eta_{ij}=
	\psi_i[\gothm_j]=\Art (\gothm_j, H_2F_{4,i}/K)
	\in \Gal(H_2F_{4,i}/H_2)=\F2.
\end{equation}
In cases where we know the kernel of $\alpha$ in \eqref{eqn:alpha},
i.e., the non-trivial relation $A_D$ between the ramified primes of $K$ in $C$,
we can use it to leave out a column of~$R_8$
corresponding to a `superfluous' generator $[\gothm_j]$
of $C[2]\cap 2C$, and work with an $(r_4\times r_4)$-matrix 
to describe $\varphi_8$ in \eqref{eqn:R8}.

A product $\gothm$ of distinct ramified primes of $K$
is characterized by the squarefree divisor $m|D$ arising as its norm,
and the residue class of a quartic character $\psi$
in $\widehat C[4]/\widehat C[2]$
by the invariant field $E=\Que(\sqrt{d_1}, \sqrt{d_2})$ of the
quadratic character $2\psi$ corresponding to a
decomposition $D=d_1d_2$ `of the second kind'.
This leads to a classical notation for the entries 
$\psi([\gothm])$ in \eqref{eqn:R8entries} as \emph{R\'edei symbols}.
\begin{definition}
\label{def:redeisym}
Let $D=d_1d_2$ be a decomposition of the second kind,
$K\subset F$ a corresponding extension as in condition $(1)$
of Lemma $\ref{lemma:F2-criterion}$, and $m|D$ the squarefree norm
of an integral ideal $\gothm$ in $K$ with $[\gothm]\in C[2]\cap 2C$.
Then the R\'edei symbol associated to $d_1$, $d_2$, and $m$ is
the Artin symbol
$$
[d_1, d_2, m]=\Art (\gothm, H_2F/K)\in \Gal(H_2F/H_2)=\F2.
$$
\end{definition}

\noindent
It is convenient to take the value of R\'edei symbols in $\F2$, as 
we do in Definition~\ref{def:redeisym}.
After all, they arise as entries $\eta_{ij}$ of an $\F2$-matrix $R_8$ in 
\eqref{eqn:R8entries}.
However, they describe the Galois action on certain square roots,
just like the entries $\varepsilon_{ij}$ of $R_4$ in \eqref{eqn:R4entries},
so their values are traditionally taken in $\{\pm1\}$.
We have been unable to completely avoid this notational ambiguity, which
already occurs in \eqref{eqn:R4entries}.
Despite our additive definition, our proof of R\'edei reciprocity in 
Section \ref{S:Proving Redei reciprocity} views R\'edei symbols in
\eqref{eqn:prodp-parts} as `products' of local symbols $[a,b,c]_p$,
which are recognized in our key lemma \ref{lemma:keylemma}
as quadratic Hilbert symbols satisfying a well-known global
\emph{product formula} that we did not rename into a sum formula.
On the other hand, our quadratic characters in \eqref{eqn:chipi}
and biquadratic characters in \eqref{eqn:biquadraticchar} 
do take values in~$\F2$.

\section{Computing Redei-symbols}
\label{S:Computing Redei-symbols}

\noindent
Definition \ref{def:redeisym} of the R\'edei symbol $[d_1, d_2, m]$
does not immediately show how to compute it
from $d_1$, $d_2$ and $m$, as it
involves a quadratic extension $F$ of
$\Que(\sqrt{d_1}, \sqrt{d_2})$ that is dihedral over $\Que$.
Galois theory tells us that such $F$ come from rational
points on the conic $x^2-d_1y^2-d_2z^2=0$, and that they are unique
up to twisting by quadratic characters.
This statement does not depend on the base field $\Que$, and 
we formulate it for any field $Q$ of characteristic different from 2.
\begin{lemma}
\label{lemma:biquadratic}
Let $Q$ be of characteristic different from $2$, and 
$Q\subset Q(\sqrt a)$ a quadratic extension.
For $\beta\in Q(\sqrt a)^*$ non-square of norm
$N\beta=b\in Q^*$, let $F$ be the normal closure of 
the quartic extension
$Q(\sqrt a,\sqrt \beta)$ of $Q$. Then
\begin{enumerate}
\item
for $\bar b\notin\{\bar 1, \bar a\}\subset Q^*/{Q^*}^2$, the field $F$
is quadratic over $Q(\sqrt a, \sqrt b)$, cyclic of degree $4$ over
$Q(\sqrt {ab})$, and dihedral of degree $8$ over $Q$;
\item
for $\bar b=\bar a\in Q^*/{Q^*}^2$, the field 
$F$ is quadratic over $Q(\sqrt a)$ and cyclic of degree $4$ over $Q$;
\item
for $\bar b=\bar 1\in Q^*/{Q^*}^2$, the field
$F$ is quadratic over $Q(\sqrt a)$ and non-cyclic abelian of degree $4$
over $Q$.
\end{enumerate}
Conversely, every field $F$ having the properties in $(1)$, $(2)$,
or $(3)$ is obtained in this way for some $\beta\in Q(\sqrt a)$ of norm $b$.
\end{lemma}
\begin{proof}
Basic Galois theory.
\end{proof}

\begin{corollary}
\label{corollary:D4}
Let $a, b\in Q^*$ and $E=Q(\sqrt a, \sqrt b)$ be as in $(1)$ of
Lemma $\ref{lemma:biquadratic}$. 
Then a quadratic extension $E\subset F$
is cyclic over $Q(\sqrt{ab})$ and dihedral of degree $8$ over $Q$
if and only if there
exists a non-zero solution $(x,y,z)\in Q^3$ to the equation
$$
x^2-a y^2-b z^2=0
$$
such that for $\beta=x+y\sqrt a\in Q(\sqrt a)$
and $\alpha=2(x+z\sqrt b)\in Q(\sqrt b)$
of norm $\beta\beta'\in b\cdot {Q^*}^2$ and $\alpha\alpha'\in a \cdot {Q^*}^2$,
we have 
$$
F=E(\sqrt \beta)=E(\sqrt\alpha).
$$
Given $F=E(\sqrt \beta)$, any other quadratic extension of $E$ that is
dihedral over $Q$ is of the form
$F_t=E(\sqrt{t\beta})$ for some unique $t\in Q^*/\langle a, b, {Q^*}^2\rangle$.
\end{corollary}
\begin{proof}
The first statement follows if we write $\beta=x+y\sqrt a\in Q(\sqrt a)$ 
in the dihedral case (1) of Lemma \ref{lemma:biquadratic}, and
observe that $F=Q(\sqrt a, \sqrt \beta, \sqrt{\beta'})=E(\sqrt \beta)$
is the normal closure over $Q$ of
$Q(\sqrt a,\sqrt \beta)$, but also of $Q(\sqrt b,\sqrt\alpha)$:
it contains a square root of the non-square element
\begin{equation}
\label{eqn:absymm}
(\sqrt\beta+\sqrt{\beta'})^2=
\left(\sqrt{x+y\sqrt a}+\sqrt{x-y\sqrt a}\right)^2=
	2(x+ z\sqrt{b})=\alpha \in \Que(\sqrt{b})^*.
\end{equation}
The dihedral group $D_4$ of order 8 has center $Z(D_4)=\FF_2$ with
quotient $D_4/Z(D_4)=\FF_2\times\FF_2$, and 
extending $Q\subset Q(\sqrt a, \sqrt b)$ 
to a $D_4$-extension amounts to lifting the surjection
$G_Q\to\FF_2\times\FF_2=D_4/Z(D_4)$ on the absolute Galois group of $Q$
corresponding to $Q(\sqrt a, \sqrt b)$ to a homomorphism $f:G_Q\to D_4$.
Given $f$ corresponding to $F=E(\sqrt \beta)$, any other lift is of
the form $f_t=\chi_t f$ for some quadratic character
$\chi_t: G_Q\to Z(D_4)=\FF_2$ corresponding to $Q\subset Q(\sqrt t)$.
The extension $F_t$ corresponding to $f_t$ is
the \emph{quadratic twist} $F_t=E(\sqrt{t\beta})$ of $F$,
and $t\in Q^*$ yielding $F_t$ is unique up to multiplication by elements of
${E^*}^2\cap Q^*=\langle a, b, {Q^*}^2\rangle$.
\end{proof}
\begin{corollary}
\label{corollary:C4}
For $E=Q(\sqrt a)$ as in Lemma $\ref{lemma:biquadratic}$, a
quadratic extension $E\subset F$ is cyclic over
$Q$ if and only if there
exists a non-zero solution $(x,y,z)\in Q^3$ to 
$$
x^2-a y^2-a z^2=0
$$
such that we have $F=E(\sqrt\alpha)$ for $\alpha=x+y\sqrt a\in Q(\sqrt a)$
of norm $\alpha\alpha'\in a\cdot {Q^*}^2$.
Given one $F=E(\sqrt \alpha)$, any other such extension is of the form
$F_t=E(\sqrt{t\alpha})$ for some unique $t\in Q^*/\langle a, {Q^*}^2\rangle$.
\end{corollary}
\begin{proof}
Analogous to the dihedral case.
\end{proof}
\begin{remark}
The dihedral group $D_4$ of order 8 can be viewed
as the Heisenberg group $U_3(\F2)$ of upper triangular
$3\times 3$-matrices with coefficients in $\F2$,
and extending an extension $Q\subset Q(\sqrt a, \sqrt b)$ to
a $D_4$-extension amounts to an embedding problem that can be
treated in terms of Massey symbols \cite{Minac}.
For our purposes, the basic Galois theory of Lemma \ref{lemma:biquadratic}
and its corollaries are already sufficient.
\hfill$\lozenge$
\end{remark}
\noindent
In order to construct an unramified extension $K\subset F$
containing $E=\Que(\sqrt{d_1}, \sqrt{d_2})$ for $D=d_1d_2$
satisfying the conditions of Lemma \ref{lemma:F2-criterion},
we apply Corollary \ref{corollary:D4} for $Q=\Que$ and $(a,b)=(d_1, d_2)$.
It shows that $F$ can be explicitly generated as 
\begin{equation}
\label{eqn:F4}
F=F(x,y,z)=E(\sqrt {\delta_2})=E(\sqrt {\delta_1}),
\end{equation}
for elements
$
\delta_2=x+y\sqrt{d_1}\in \Que(\sqrt{d_1})^*
$
and 
$
\delta_1=2(x+y\sqrt{d_1})\in \Que(\sqrt{d_1})^*
$
coming from a solution $(x,y,z)$ to the equation
\begin{equation}
\label{eqn:normd1d2}
x^2-d_1y^2-d_2 z^2=0.
\end{equation}
By Corollary \ref{corollary:D4},
scaling \emph{any} non-zero solution $(x,y,z)$ with an
appropriate element $t\in\Que^*$, which amounts to replacing 
$F(x,y,z)$ by the \emph{quadratic twist} $F(tx,ty,tz)$,
will make $K\subset F$ unramified.
As we will show in a slightly more general setting in 
Corollary \ref{cor:F4unr}, every \emph{primitive}
integral solution $(x,y,z)$ to \eqref{eqn:normd1d2} yields
an extension $K\subset F(x,y,z)$ that is unramified at all odd primes.
Ramification over 2 can be avoided by twisting the extension
with a suitable choice of $t\in \{\pm1, \pm2\}$.

\begin{example}
\label{example:820}
Take $K=\Que(\sqrt{-205})$ of discriminant $D=-4\cdot5\cdot41=-820$,
which has $t=3$ and $r_2=2$. 
The columns of the R\'edei matrix
$$
R_4=
\begin{pmatrix}1 & 0 & 0\\ 1 & 0 & 0 \\ 0 & 0 &0\end{pmatrix}
$$
describe the action of the Artin symbols of the 
three ramified primes $\gothp_2$, $\gothp_5$, and~$\gothp_{41}$
dividing $D$ on the square roots of $-4$, 5 and 41 generating
$H_2=\Que(i, \sqrt 5, \sqrt {41})$ as in \eqref{eqn:R4entries}.
From the matrix $R_4$ we read off that $r_4$ equals
$r_2-\rank(R_4)=1$, that $[\gothp_5]$ and $[\gothp_{41}]$ span $C[2]\cap 2C$,
and that $D=-20\cdot41$ is the unique decomposition of the second kind.
The equation
$$
x^2+20y^2-41z^2=0
$$
has a primitive solution (12,1,2) for which the element
$\delta=12+2\sqrt{-5}$ of norm $2^2\cdot 41$ is `primitive outside 2' and
satisfies $\delta\congr(1+\sqrt{-5})^2\mod 4$. 
This shows that $\delta=2(6+\sqrt{-5})$ has
an unramified square root over $E=\Que(\sqrt{-5}, \sqrt{41})$,
whereas the primitive elements $\pm6+\sqrt{-5}$ yield
extensions $E\subset E(\sqrt{\pm\delta/2})$ that are ramified over~$2$.

The solution $(17,2,3)$ defining the primitive element
$\delta_0=17+4\sqrt{-5}$ of norm $3^2\cdot 41$ satisfying
$\delta_0\congr1\mod4$ also has an unramified square root over $E$.
We have
$$
\delta\delta_0
	= 164+82\sqrt{-5}
	= -[\sqrt{41}(1-\sqrt{-5})]^2\in -1\cdot {E^*}^2,
$$
and $E(\sqrt{\delta_0})=E(\sqrt{t\delta})$ for $t=-1$.
Over $H_2$, both $\sqrt{\delta_0}$ and $\sqrt{\delta}$ generate
$$
H_4=H_2(\sqrt{\delta_0})=H_2(\sqrt{\delta}).
$$
As we know that $(\sqrt{-5\cdot 41})=\gothp_5\gothp_{41}$ is trivial in $C$,
the class of either $\gothp_5$ or $\gothp_{41}$ generates $C[2]\cap 2C$.
The matrix $R_8$ consists of a single R\'edei symbol
$$
[-20, 41, 5]=[-20,41, 41]
$$
describing whether the prime $\gothp_5$ (or, equivalently, $\gothp_{41}$)
of $K$ splits completely in~$H_4$.
It does not, as $\delta=12+2\sqrt{-5}$ (like $\delta_0=17+4\sqrt{-5})$)
is congruent to the quadratic non-residue 2 modulo every prime over 5 in $H_2$.
We conclude that we have $r_8=0$, and that the 2-part of $C$ is isomorphic
to $\Zee/2\Zee\times \Zee/4\Zee$.

In this case, the decomposition $D=d_1d_2=-4\cdot 205$ is not
of the second kind, but the conic
$$
x^2+4y^2-205z^2=0
$$
defined by \eqref{eqn:normd1d2} does have infinitely many
rational points $(x,y,z)$, such as $(3,7,1)$.
None of them defines an unramified quartic extension $K\subset F(x,y,z)$.
\hfill$\lozenge$
\end{example}

\noindent
Example \ref{example:820} shows that
non-trivial solvability of \eqref{eqn:normd1d2} over~$\Que$
may not guarantee the existence of unramified extensions $K\subset F(x,y,z)$,
whereas the slightly stronger conditions of Lemma \ref{lemma:F2-criterion} do.
More precisely, by the classical \emph{local-global principle\/} for conics,
assuming solvability of \eqref{eqn:normd1d2}
amounts to having quadratic Hilbert symbols $(d_1,d_2)_p=1$
for all finite primes $p$, with $(d_1,d_2)_\infty=1$ then being implied by
the product formula.
At $p\nmid D=d_1d_2$, including $p=2$, we have $(d_1,d_2)_p=1$.
For an \emph{odd} prime $p$ dividing $D=d_1d_2$, say $d_1$, 
the Hilbert symbol condition at $p$ is
$$
(d_1,d_2)_p=\Legendre{d_2}{p}=1,
$$
exactly as in condition (4) of Lemma \ref{lemma:F2-criterion}.
For $p=2$ dividing $d_1$ however, we have $d_2\congr1\mod4$ and obtain
\begin{equation}
\label{eqn:-4*5}
(d_1,d_2)_2=\Legendre{d_2}{2}^{\ord_2(d_1)}=1.
\end{equation}
If $D$ has 2-part $t_D=-4$ in \eqref{eqn:Dfactorization}
and we have $d_2\congr 5\mod 8$, such as for $D=-820$ and $d_2=205$ above,
condition \eqref{eqn:-4*5} does \emph{not} imply condition (4) of 
Lemma \ref{lemma:F2-criterion} for $p=2$, as we have 
$(d_1,d_2)_2=1$ but $\Legendre{d_2}{2}=-1$.
This discrepancy will lead us to the notion of 2-minimal ramification 
in Definition \ref{def:2minram}.

\section{Discovering Redei reciprocity}
\label{S:Discovering Redei reciprocity}

\noindent
Explicit computations of R\'edei symbols exhibit a `reciprocity law'
that we can discover already by looking at the most classical
example 
$$
K=\Que(\sqrt {-p}),
$$
with $p$ an odd prime.
Having $r_2=t-1=1$ for $K$ amounts to having $p\congr 1\mod 4$ and $D=-4p$
by Theorem \ref{theorem:2-rank}; otherwise $C$ has odd order.
Assuming this, $R_4$ in Theorem \ref{theorem:4-rank}
has rank 0 if and only if we have $\legendre{p}{2}=\legendre{2}{p}=1$,
so $r_4=1$ happens for $p\congr 1\mod 8$.
We further assume $p\congr 1\mod 8$.

As $(\sqrt{-p})|p$ is principal,
the class of the non-principal prime $\gothp_2|2$ generates $C[2]=C[2]\cap 2C$,
and $R_8$ consists of a single R\'edei symbol $[-4,p,2]$.
The case $r_8=1$ in which the symbol vanishes
occurs when the prime $\gothp_2$ of $K$
splits completely in the 4-Hilbert class field $H_4(-p)$ of $K$. 
Solving
$$
x^2+4y^2-pz^2=0
$$
with $z=1$, we can generate $H_4(-p)$ over
$E=\Que(i, \sqrt p)$ by adjoining a square root of $\pi=x+2iy$.
Now $\gothp_2$ splits into 4 primes in the extension $K\subset H_4(-p)$ if and 
only if the prime $(1+i)$ over 2 in $\Que(i)$ splits into 4 primes
in the extension 
$$
\Que(i)\subset H_4(-p)=\Que(\sqrt{-p}, i, \sqrt\pi)=\Que(i, \sqrt \pi,\sqrt{\bar\pi}).
$$
This shows that $[-4,p,2]$ can be viewed as
a `Kronecker symbol' $\legendre{\pi}{1+i}$ in $\Que(i)$.
By class field theory (or quadratic reciprocity) over $\Que(i)$,
this symbol is simply the Legendre symbol
$\legendre{1+i}{p}$, which is well defined for $p\congr 1\mod 8$,
and we have 
\begin{equation}
r_8=1\quad \Longleftrightarrow\quad
\text{$p$ splits completely in $\Que(\zeta_8,\sqrt{1+i})$.}
\end{equation}
We deduce that the prime $\gothp_2$ over 2 splits completely in
the unramified extension $K\subset H_4(-p)$
if and only if $p$ splits completely in $\Que(\zeta_8,\sqrt{1+i})$.
By the case $(a,b)=(-1,2)$ of Lemma \ref{lemma:biquadratic}, the field
$\Que(\zeta_8,\sqrt{1+i})= \Que(i, \sqrt 2,\sqrt{1+i})$
is dihedral over $\Que$, just like $H_4(-p)$.
In fact, both fields are abelian of exponent 2 over $\Que(i)$,
quadratic over
respectively $\Que(i,\sqrt 2)$ and $\Que(i, \sqrt p)$, and cyclic over
respectively $\Que(\sqrt {-2})$ and $\Que(\sqrt{-p})$.
We have proved a special case of \emph{R\'edei reciprocity}, and in terms of 
R\'edei symbols it can be suggestively formulated as
\begin{equation}
[-4,p,2]=[-4,2,p].
\end{equation}
The symbol on the left is defined by Definition \ref{def:redeisym}
for symbols $[d_1,d_2,m]$, but the symbol on the right is not.
It is natural to take $d_1$ and $d_2$ up to squares, yielding
the formulation $[-1,p,2]=[-1,2,p]$,
but the symbol $[-1,2,p]$ refers to the splitting of the primes over $p$
in 
$$
\Que(\sqrt{-2})\subset E=\Que(i,\sqrt 2)\subset F=E(\sqrt{1+i}),
$$
a cyclic quartic extension that is totally ramified over 2.
Primes $p\congr 1 \mod 8$ are totally split in $E$, and
split or inert in $E\subset F$ depending on the value of $[-1,2,p]$.

As we can swap the arguments $-1$ and $p$ in the symbol $[-1,p,2]$ by 
its very definition, and $2$ and $p$ by what we just
proved, one naturally wonders whether it also equals
a symbol $[2,p,-1]$ that describes the splitting of ``$-1$"
in the narrow 4-Hilbert class field $H_4(2p)$ of $\Que(\sqrt{2p})$.
By Theorem \ref{theorem:4-rank}, the field $H_4(2p)$
is quadratic over the totally real field $\Que(\sqrt 2,\sqrt p)$
for $p\congr1\mod 8$.
Now Frobenius symbols at ``$-1$", which over $\Que$ raise roots of
unity to the power $-1$, arise in class field theory as complex conjugations,
and act trivially on totally real fields.
The dihedral field $H_4(2p)$ is abelian 
of exponent 2 over $\Que(\sqrt 2)$, and
it is totally real if and only if its conductor over $\Que(\sqrt 2)$
is $p$, not $p\cdot\infty$.
Looking at the ray class group
$$
(\Zee[\sqrt2]/p\Zee[\sqrt2])^*/\langle -1, 1+\sqrt 2\rangle
$$
modulo $p$ of $\Que(\sqrt 2)$, we see
that $H_4(2p)$ is real exactly when the 
fundamental unit $1+\sqrt 2\in \Que(\sqrt 2)$ is a square modulo $p$,
and this happens for the primes that split completely in
the dihedral field $\Que(\zeta_8, \sqrt{1+\sqrt 2})$.
By equation \eqref{eqn:absymm} for $(a,b)=(-1,2)$ and $(x,y,z)=(1,1,1)$,
this is the \emph{same} field as $\Que(\zeta_8,\sqrt{1+i})$, so
the R\'edei symbol 
\begin{equation}
\label{eqn:12p}
[-1,p,2]=[-1,2,p]=[2,p,-1],
\end{equation}
when properly defined, is invariant under
\emph{all permutations} of its arguments.

In R\'edei's own definition, $[-1,2,p]$ does not exist, and 
$[p,2,-1]$ is trivial for all~$p$. 
Our definition in the next section introduces a notion of 
\emph{minimal ramification} for extensions $K\subset F$ as
in \eqref{eqn:F4}, correcting the definition found in \cite{Corsman}.

\section{Redei symbols}
\label{S:Redei symbols}

\noindent
In order to obtain R\'edei reciprocity,
we generalize the symbol $[d_1, d_2, m]$
in Definition \ref{def:redeisym} beyond the setting
of dihedral fields $F$ containing $\Que(\sqrt{d_1},\sqrt{d_2})$
that are cyclic and unramified over $K=\Que(\sqrt{d_1d_2})$ and
norms $m$ of ambiguous ideals $\gothm$ of $K$ with trivial
Artin symbol in the genus field of $K$.
As $d_1$ and $d_2$ encode quadratic fields, and $m$ is the norm of
an ideal with ideal class in a subgroup of $C$ of exponent~2,
the general R\'edei symbol $[a,b,c]$
naturally takes its arguments in 
the group $\Que^*/{\Que^*}^2$.
It will be linear in each of its arguments.

Every $\bar a\in\Que^*/{\Que^*}^2$ is uniquely represented by a 
squarefree integer $a$, and corresponds to a number field
$\Que(\sqrt a)$ that is quadratic for $a\ne 1$.
Given non-trivial elements in $\Que^*/{\Que^*}^2$
represented by squarefree integers $a$, $b$, the extension
\begin{equation}
\label{eqn:KsubsetE}
\Que(\sqrt{ab})=K\subset E=\Que(\sqrt a, \sqrt b)
\end{equation}
is quadratic, and $\Que\subset E$ is unramified at primes outside
the discriminants $\Delta(a)$ and $\Delta(b)$ of the quadratic fields
corresponding to $a$ and $b$. 

We now assume that $a$ and $b$ have relative quadratic Hilbert symbols
$(a,b)_p=1$ for all primes $p$.
As we observed at the end of Section \ref{S:Computing Redei-symbols},
this amounts to saying that the equation
\begin{equation}
\label{eqn:normab}
x^2-a y^2-b z^2=0
\end{equation}
admits non-zero rational solutions.
By Corollaries \ref{corollary:D4} and \ref{corollary:C4},
such a solution $(x,y,z)$ generates a cyclic quartic extension
\begin{equation}
\label{eqn:Fxyz}
\Que(\sqrt{ab})=K\subset F=F(x,y,z)=E(\sqrt \beta)=E(\sqrt\alpha)
\end{equation}
in which we take $\beta=x+y\sqrt a$ and $\alpha=2(x+z\sqrt b)$.
The field $F$ is dihedral over~$\Que$ for $a\ne b$, and cyclic over
$\Que(\sqrt{ab})=\Que$ for $a=b$.
It is uniquely determined by $a$ and $b$ up to
twisting by rational quadratic characters $\chi_t$, with
$t\in\Que^*/{\Que^*}^2$.
In fact, the asymmetry in \eqref{eqn:Fxyz} in the definition of
$\alpha$ and $\beta$ coming out of \eqref{eqn:normab}
can be seen as making a somewhat arbitrary choice between $F$ and $F_2$.
Here we use the twisting notation
\begin{equation}
\label{eqn:twist}
F_t=F(tx,ty,tz)=E(\sqrt {t\beta})=E(\sqrt{t\alpha})
\end{equation}
for $F=F(x,y,z)$ from Section \ref{S:Computing Redei-symbols}.

Before defining the general symbol $[a,b,c]$, we start
with the special case of a squarefree integer $a=b\ne1$.
Then $K\subset E$ in \eqref{eqn:KsubsetE} is the extension
$\Que\subset \Que(\sqrt a)$.
As Hilbert symbols satisfy $(a,-a)_p=1$ for all~$a\in\Que^*$,
our assumption on $a$ is $(a,-1)_p=1$ for all $p$.
By~\eqref{eqn:-1norm}, this means that the discriminant 
$\Delta(a)\in\{a,4a\}$ is positive and without
prime divisors $3\mod4$, i.e.,
in the set $\calD$ of Conjecture \ref{conj:negpell}.
For such $a$, we can write the associated Dirichlet character 
$\chi_a:G_\Que\to\FF_2$
in the notation of \eqref{eqn:chiD} as 
$$
\chi_a=\sum_{p|a} \chi_p,
$$
with $\chi_2=\chi_8$ the quadratic character
associated to $\Que(\sqrt 2)$.
Let $\psi_p$ be a character of order 4 modulo $p\congr1\mod 4$,
and $\psi_2$ a character of order 4 on~$(\Zee/16\Zee)^*$.
Then
\begin{equation}
\label{eqn:psia}
\psi_a=\sum_{p|a} \psi_p
\end{equation}
is a quartic Dirichlet character of conductor $a$ (or $8a$ when $a$ is even)
that satisfies $2\psi_a=\chi_a$.
It corresponds to a cyclic quartic field
$F_{a,a}$ containing $E=\Que(\sqrt a)$. 
Cyclic quartic $\Que\subset F$ containing $E$
are unique up to quadratic twists, as $2\psi_F=\chi_a$ 
only defines $\psi_F$ up to a quadratic character.
Clearly $E\subset F$ will be ramified at all primes dividing $a$,
but not at other primes if we take $\psi_F=\psi_a$ as in \eqref{eqn:psia}.
Such an extension $\Que\subset F_{a,a}$ of \emph{minimal ramification}
is not unique, as we can add $\chi_p$ to
each $\psi_p$ of odd conductor $p$ in \eqref{eqn:psia}, and
$\chi_2$ and $\chi_{-1}$ to $\psi_2$.
This makes $\psi_a$ unique up to twisting by sums
of characters $\chi_p$ with $p|a$, and in addition the character
$\chi_{-1}$ in case $a$ is even.
For $a=1$ we define $\psi_a$ to be the trivial character.

In terms of Corollary \ref{corollary:C4}, a
minimally ramified cyclic quartic extension $\Que\subset F_{a,a}$
containing $E=\Que(\sqrt a)$ is unique up to twisting by $t$
in the finite \emph{twisting subgroup}
\begin{equation}
\label{eqn:Taa}
T_{a,a}\subset \Que^*/{\Que^*}^2
\end{equation}
generated by the residue classes of the odd primes $p=p^*$
dividing $a$ and, for $a$ even, of $-1$ and 2.
It follows that for squarefree integers $a$ and $c$ satisfying
\begin{equation}
\label{eqn:aachyp}
\gcd(\Delta(a),\Delta(c))=1
\quad\hbox{and}\quad
\hbox{$(a,a)_p=(a,c)_p=1$ for all $p$,}
\end{equation}
we have $\chi_t(c)=0$ for $t\in T_{a,a}$, and
a well defined \emph{biquadratic character}
\begin{equation}
\label{eqn:biquadraticchar}
\psi_a(c)=\Legendre{c}{a}_4\in \FF_2.
\end{equation}
In terms of Artin symbols in the cyclotomic field $\Que(\zeta_a)$
(or $\Que(\zeta_{8a})$) containing $F_{a,a}$, we have
\begin{equation}
\label{eqn:psiac}
\psi_a(c)= \Art(c, F_{a,a}/\Que)\in\Gal(F_{a,a}/E)=\FF_2.
\end{equation}
Note that $a$ and $\Delta(a)\in\calD$ have the same prime factors
by the hypothesis $(a,a)_p=1$ for all $p$.
Thus, \eqref{eqn:aachyp} implies that for $a$ even,
with $-1, 2\in T_{a,a}$, we have 
$c\congr 1\mod 4$ by the gcd-condition, and then $c\congr 1\mod 8$
by the condition $(a,c)_2=1$, yielding $\chi_2(c)=\chi_{-1}(c)=0$.
For odd primes $p$ dividing~$a$, with $p\in T_{a,a}$, 
the conditions $(a,c)_p=1$ and $\chi_p(c)=0$ coincide.

Viewing $\psi_a$ in \eqref{eqn:biquadraticchar} as defined
on the subgroup of $\Que^*/{\Que^*}^2$ generated by the integers $c$
satisfying \eqref{eqn:aachyp}, we obtain the following special 
R\'edei symbol.
\begin{definition}
\label{def:aac}
For $a, c \in \Que^*/{\Que^*}^2$ satisfying \eqref{eqn:aachyp},
we take \eqref{eqn:biquadraticchar} to define
$$
[a,a,c]=\psi_a(c)=\Legendre{c}{a}_4\in \FF_2.
$$
\end{definition}

\noindent
For the general case, we
take $a, b\in\Zee_{\ne1}$ to be different squarefree integers,
assuming 
\begin{equation}
\label{eqn:abp=1}
(a,b)_p=1\quad\hbox{for all primes $p$}
\end{equation}
in order to have non-trivial solvability of \eqref{eqn:normab}.
Then the quadratic extension
$$
\Que(\sqrt{ab})=K\subset E=\Que(\sqrt a, \sqrt b)
$$
from \eqref{eqn:KsubsetE}
is only ramified at primes $\gothp | \gcd(\Delta(a),\Delta(b))$,
and such $\gothp$ will be totally ramified in every cyclic quartic extension 
$K\subset F$ in \eqref{eqn:Fxyz}.
Generalizing \eqref{eqn:psiac}, we are to define the R\'edei symbol 
in Definition \ref{def:redeisymbol} as
$$
[a,b,c]=\Art_c (F_{a,b}/K)\in \Gal(F_{a,b}/E)=\FF_2,
$$
the Artin symbol of an ideal $\gothc$ in $K$ corresponding to $c$
in a cyclic quartic extension $K\subset F_{a,b}$
constructed from a solution to \eqref{eqn:normab} that is
\emph{minimally ramified} over $E$.

For \emph{odd} primes $p\nmid \gcd(\Delta(a),\Delta(b))$,
one can avoid ramification over $p$ in $E\subset F$
by passing, if needed,
to the quadratic twist $F_{p^*}$ from \eqref{eqn:twist}, with
$p^*=\pm p$ as in \eqref{eqn:Dfactorization}.

\begin{proposition}
\label{prop:F4unramatp}
Let $a,b\in\Zee_{\ne1}$ be distinct and squarefree, $p\nmid\Delta(b)$
an odd prime, and
$K=\Que(\sqrt{ab})\subset F=E(\sqrt\beta)$ as in \eqref{eqn:Fxyz}.
\begin{enumerate}
\item
If $p$ divides $\Delta(a)$, then $K\subset F$ is unramified over $p$.
\item
If $p$ does not divide $\Delta(a)$, then exactly one of 
$K\subset F$ and $K\subset F_{p^*}$ is unramified over $p$.
\end{enumerate}
\end{proposition}
\begin{proof}
Consider $F$ as a quartic extension of $K_a=\Que(\sqrt a)$.
The intermediate field $E=K(\sqrt b)=K_a(\sqrt b)$
is a quadratic extension of both $K$ and $K_a$
that is unramified at primes dividing $p$, as we have $p\nmid \Delta(b)$.
It follows that $K\subset F$ is unramified over~$p$ if and only if
$K_a\subset F$ is.

Write $F=K_a(\sqrt \beta, \sqrt{\beta'})$, with $\beta\in K_a$ of norm
$\beta\beta'=x^2-ay^2=bz^2$.
As $p$ is odd, $K_a\subset F$ is unramified over $p$ if and only if
$\beta$ and $\beta'$ have even valuation at the primes $\gothp|p$ in $K_a$.
For a prime $\gothp|p$ of ramification index $e_{\gothp/p}$ in $K_a$, we have
$$
\ord_\gothp(\beta\beta')=e_{\gothp/p}\ord_p(bz^2)=2e_{\gothp/p}\ord_p(z).
$$
In the ramified case $p|\Delta(a)$, we have $e_{\gothp/p}=2$ and
$\ord_\gothp(\beta\beta')=2\ord_\gothp(\beta)\congr 0\mod 4$, proving (1).
In the unramified case $p\nmid\Delta(a)$, we have $e_{\gothp/p}=1$
and $\ord_\gothp(\beta)\congr \ord_\gothp(\beta')\mod 2$.
Moreover, we have
$\ord_\gothp(p^*\beta)=\ord_\gothp(\beta)+1$, so $\gothp$
is unramified in exactly one of $F$ and $F_{p^*}$, proving (2).
\end{proof}

\noindent
As twisting by $p^*$ does not change ramification outside $p$,
Proposition \ref{prop:F4unramatp} shows that $K\subset F$ in \eqref{eqn:Fxyz}
can be chosen to be unramified at all odd $p\nmid \gcd(\Delta(a),\Delta(b))$.
For $p=2$ we can twist by $-1$ and $2$, with the following outcome.
\begin{proposition}
\label{prop:F4unramat2}
Let $a,b\in\Zee_{\ne1}$ be distinct and squarefree,
$b=\Delta(b)\congr1\mod 4$, and
$K=\Que(\sqrt{ab})\subset F=E(\sqrt \beta)$ as in \eqref{eqn:Fxyz}.
\begin{enumerate}
\item
If $\Delta(a)$ is odd, then $\Que\subset F_t$ is unramified at $2$
for a unique $t\in\{\pm1,\pm2\}$.
\item
If $\Delta(a)$ is even and $\Delta(b)$ is $1\mod 8$, then $K\subset F_t$
is unramified over $2$ for exactly two values of $t\in\{\pm1,\pm2\}$.
\item
If $\Delta(a)$ is even and $\Delta(b)$ is $5\mod 8$, then $K\subset F_t$
is ramified over $2$ for all $t\in\Que^*$, and $\Delta(a)$ is $4\mod 8$.
\end{enumerate}
\end{proposition}
\begin{proof}
Just as for odd $p$,
we consider the extension $K_a\subset F=K_a(\sqrt \beta, \sqrt{\beta'})$.
As before, up to squares, $\beta$ is a 2-unit 
in $K_a$ if $\Delta(a)$ is even, and exactly one of $\beta$ and $2\beta$
is a 2-unit $K_a$ if $\Delta(a)$ is odd.
However, for a 2-unit to have a square root that is unramified at~2,
we need the stronger condition that it is a square modulo~4.

We can assume, possibly after twisting $F$ by $t=2$, that
$\beta$ is a 2-unit in the ring of integers $\Ocal$ of $K_a$.
For $2\nmid \Delta(a)$, the group $(\Ocal/4\Ocal)^*$ has order 4 or 12,
depending on whether 2 is split or inert in $\Ocal$, and the squares
form a subgroup of index 4.
Together with $-1$, they generate the kernel of the surjective norm map
$$
N: (\Ocal/4\Ocal)^*\tto (\Zee/4\Zee)^*.
$$
By assumption,
we have $\beta\beta'\congr b\congr1\mod 4\Ocal$, so the residue classes 
$\beta, \beta'\in\ker  N$ are squares in $(\Ocal/4\Ocal)^*$ for a
unique `sign choice' of $\beta$, and 
$\Que\subset F_t$ is unramified at 2 for a unique value 
$t\in\{\pm1,\pm2\}$. This proves (1).

For $2|\Delta(a)$, the group $(\Ocal/4\Ocal)^*=(\Ocal/\gothp^4\Ocal)^*$
has order 8, and its subgroup of squares, of index 4, is of order 2.
The norm $\Ocal=\Zee[\sqrt a]\to \Zee$ induces a map
$$
N: (\Ocal/4\Ocal)^*\tto (\Zee/8\Zee)^*
$$
for which the image, of order 2, is generated by $1-a\mod 8$ when 
$a\congr\pm2\mod 8$ is even, and by $5\mod 8$ when $a\congr -1\mod 4$ is odd.

In the case where $a$ is even, $\ker  N$ is non-cyclic of order 4,
generated by $-1$ and the squares in $(\Ocal/4\Ocal)^*$, and
it contains $\beta\mod 4\Ocal$ as $\beta\beta'\congr b\mod 8$
is not $5\mod 8\notin\im N$. 
In this case, we have $\Delta(b)= b\congr1\mod 8$,
and we conclude just as before that exactly one of
$F$ and $F_{-1}$ is unramified over $K$ at 2.
By the same argument applied to $F_2$, one of $F_2$ and $F_{-2}$
is unramified over $K$ at~2, so $K\subset F_t$ is unramified over 2
for exactly two values $t\in\{\pm1, \pm2\}$, as stated in~(2).

In the remaining case $a\congr -1\mod 4$, or $\Delta(a)\congr 4\mod 8$,
the residue class of 
\begin{equation}
\label{eqn:tau}
\tau=(1+\sqrt a)^2/2=(1+a)/2+\sqrt a
\end{equation}
in $(\Ocal/4\Ocal)^*$, which equals $\sqrt a\mod4\Ocal$ for
$a\congr -1\mod 8$ and $2+\sqrt a\mod4\Ocal$ for $a\congr 3\mod 8$,
has square $-1\mod4\Ocal$, so it is of order 4 and generates $\ker  N$.

We now have 2 cases.
For $\Delta(b)= b\congr1\mod 8$ we have $\beta\mod 4\Ocal\in\ker  N$,
and twisting by $t=2$, which replaces $\beta$ by $\beta/\tau$, may be
used to move $\beta$ into the subgroup $\pm1\mod4\Ocal$ of squares in
$(\Ocal/4\Ocal)^*$.
In this case either $F$ and $F_{-1}$ or $F_2$ and $F_{-2}$ are
unramified over $K$ at 2, proving (2).

The final case $a\congr-1\mod4$ and
$\Delta(b)= b\congr5\mod 8$ is the case occurring in~(3).
Here twisting by $-1$ or $2$ cannot move $\beta$ or $\beta'$ into
$\ker N$, and the extension $K_a\subset F=K_a(\sqrt\beta,\sqrt{\beta'})$
is ramified at the prime $\gothp|2$ of $K_a$.
This implies that $K\subset F_t$ is ramified over 2 for all
$t\in\Que^*$, proving (3).

Alternatively, one can argue for (3) that if the ramified prime over 2 in $K$,
which is inert in $K\subset E$,
were unramified in the cyclic quartic extension $K\subset F_t$,
the primes over 2 in $F_t$ would have ramification index 2
and residue class degree~4 over~$\Que$;
but the dihedral group of order 8 has no cyclic quotient of order 4.
\end{proof}
\noindent
The \emph{ramified case} (3) of Proposition \ref{prop:F4unramat2} does not occur when
$D=\Delta(a)\Delta(b)$ is a decomposition satisfying the conditions of
Lemma~\ref{lemma:F2-criterion}, as for even $D$, the prime 2
splits in either $\Que(\sqrt a)$ or $\Que(\sqrt b)$, by 
condition (4) of Lemma~\ref{lemma:F2-criterion}.
\begin{corollary}
\label{cor:F4unr}
Let $(x,y,z)$ be a primitive integral solution to \eqref{eqn:normd1d2}
for $D=d_1d_2$ satisfying the conditions of Lemma~$\ref{lemma:F2-criterion}$.
Then there exists $t\in\{\pm1,\pm2\}$ such that 
$F_t=\Que(\sqrt{d_1},\sqrt{d_2}, \sqrt{tx+ty\sqrt {d_1}})$ is 
unramified and cyclic of degree $4$ over $\Que(\sqrt D)$.
\end{corollary}
\begin{proof}
For $(x,y,z)$ primitive and $p$ odd, $\beta=x+y\sqrt{d_1}$ and
$\alpha=2(x+z\sqrt{d_2})$ are not divisible by $p$,
hence units at a prime over $p$ in $\Que(\beta)$ and $\Que(\alpha)$,
making $\Que(\sqrt D)\subset F_1$ unramified outside 2.
Twisting by $t\in\{\pm1,\pm2\}$ as in (1) and~(2) of
Proposition \ref{prop:F4unramat2} makes it unramified at 2 as well.
\end{proof}
\noindent
In the ramified case (3) of Proposition \ref{prop:F4unramat2},
with $a\congr-1\mod 4$ and $b\congr5\mod8$,
which is essential for R\'edei reciprocity,
the extension $K\subset F$ in \eqref{eqn:Fxyz} gives rise
to a local field $F\otimes \Que_2$ that is dihedral of degree 8 
over $\Que_2$, and quadratic~over
\begin{equation}
\label{eqn:E2}
E\otimes \Que_2=\Que_2(\sqrt a,\sqrt b)=\Que_2(i, \sqrt 5).
\end{equation}
It is cyclic over $\Que_2(\sqrt{-5})$ for $a\congr-1\mod8$,
and cyclic over $\Que_2(i)$ for $a\congr -5\mod8$.
Ramification in $E\otimes \Que_2\subset F\otimes \Que_2$ cannot be
avoided, but one can obtain \emph{$2$-minimal ramification}
after twisting, if necessary, by the generator $t=2$ of 
$$
\Que_2^*/\langle a, b, {\Que_2^*}^2\rangle=
\Que_2^*/\langle -1, 5, {\Que_2^*}^2\rangle\iso\Zee/2\Zee.
$$
In view of \eqref{eqn:tau}, 
this amounts to replacing $\beta$ by $\tau\beta$.
In this way we can make $\beta$ trivial in the group
$(\Ocal/2\Ocal)^*=\langle\bar\tau\rangle=\langle\sqrt a\mod2\Ocal\rangle$
of order 2, and we can even change the sign of $\beta$ -- this does
not change $F\otimes \Que_2$ -- to achieve
$\beta\congr1\mod\gothp^3$, with $\gothp|2$ in $K_a$.
This is not quite the congruence $\beta\congr1\mod\gothp^4$
that would make $K_a=\Que(\sqrt a)\subset F$ unramified over 2, but
it does ensure that the local extension
$\Que_2(\sqrt a)\subset F\otimes \Que_2$ is of conductor $2$,
the minimum for a ramified biquadratic extension of $\Que_2(\sqrt a)$.
One has $F\otimes \Que_2=\Que_2(i, \sqrt 5, \sqrt x)$
with $x=1+2i$ for $a\congr-1\mod8$ and $x=3+2\sqrt{-5}$ for $a\congr-5\mod 8$.
\begin{definition}
\label{def:2minram}
In the ramified case $(3)$ of Proposition $\ref{prop:F4unramat2},$
with $a\congr-1\mod4$ and $b\congr5\mod8$, the extension
$K\subset F$ is $2$-minimally ramified if the local biquadratic extension
$\Que_2(\sqrt a)\subset F\otimes \Que_2$ is of conductor~$2$.
\end{definition}
\noindent
The requirement in Definition \ref{def:2minram} means that we
have $F=E(\sqrt\beta)$ for an element $\beta\in 1+2\Ocal\subset K_a^*$.
Any $F$ in case (3) of Proposition \ref{prop:F4unramat2}
has a twist $F_t$ with $t\in\{\pm1,\pm2\}$,
unique up to sign, that is $2$-minimally ramified.

For arbitrary non-trivial elements $a, b\in\Que^*/{\Que^*}^2$
for which \eqref{eqn:normab} admits non-zero solutions, we are led
to the following global notion of \emph{minimal ramification}.
\begin{definition}
\label{def:minram}
For $a, b\in\Que^*/{\Que^*}^2\setminus\{1\}$, the extension
$K\subset F$ in \eqref{eqn:Fxyz} defined 
by a non-zero rational solution to \eqref{eqn:normab}
is said to be minimally ramified over $E$ if it is
\begin{enumerate}
\item
unramified over all odd primes $p\nmid\gcd(\Delta(a),\Delta(b))$;
\item
unramified over $2$ when $\Delta(a)\Delta(b)$ is odd,
or one of $\Delta(a), \Delta(b)$ is $1\mod 8$;
\item
$2$-minimally ramified if $(\Delta(a),\Delta(b))$ is
$(5,4)$ or $(4,5)$ modulo $8$.
\end{enumerate}
\end{definition}
\noindent
In the special case $a=b$, we recover our earlier definition 
of a minimally ramified extension $\Que\subset F_{a,a}$, as
being unramified at the primes $p\nmid a$.

Every extension $K\subset F$ in \eqref{eqn:Fxyz} can be twisted by some
$t\in\Que^*/{\Que^*}^2$ to obtain a minimally ramified
extension $K\subset F_{a,b}$, but $F_{a,b}$ is not uniquely determined
by $a,b \in \Que^*/{\Que^*}^2$.
More precisely, we have a finite {\it twisting subgroup\/} 
\begin{equation}
\label{eqn:Tab}
T_{a,b}\subset \Que^*/{\Que^*}^2
\end{equation}
just as for $a=b$ in \eqref{eqn:Taa}. 
It is generated by the residue classes of the odd signed primes 
$p^*$ occurring in the discriminantal factorizations
\eqref{eqn:Dfactorization} of $\Delta(a)$ and $\Delta(b)$, 
together with $-1$ and $2$ if both $\Delta(a)$ and $\Delta(b)$
are even, and with the unique non-trivial
discriminantal 2-part $t_{\Delta(a)}$ or $t_{\Delta(b)}$
in $\{-4, \pm 8\}$ if only one of them is even.
For $a=b$, this definition coincides with \eqref{eqn:Taa}.
It is tailored to obtain the following.
\begin{lemma}
\label{lemma:twistedF}
For $a, b \in\Que^*/{\Que^*}^2\setminus\{1\}$ satisfying
\eqref{eqn:abp=1}, there exists
$F=F(x, y,z)$ in~\eqref{eqn:Fxyz}
that is minimally ramified over $E$.
For such $F$ and $t\in\Que^*/{\Que^*}^2$, we have
$$
K \subset F_t \text{\ is minimally ramified}
\equi
t\in T_{a,b}.
$$
\end{lemma}
\begin{proof}
We already showed existence.
If $\Delta(a)$ and $\Delta(b)$ are not both even, 
it follows from \eqref{eqn:chipi} that
the elements $t\in T_{a,b}$ are exactly the Dirichlet
characters of the quadratic extensions
$\Que\subset \Que(\sqrt t)$ that become unramified over 
$E=\Que(\sqrt a, \sqrt b)$,
and preserve the minimal ramification of $F$ under twisting.
If both $\Delta(a)$ and $\Delta(b)$ are even, inclusion of
both generators $-1$ and $2$ `at 2' ensures that for $t_a=t_b\ne 1$, when
Definition \ref{def:minram} imposes no restriction on ramification at 2
on $K\subset F$, we do allow all possible quadratic
twists of 2-power conductor.
\end{proof}

\noindent 
We are now ready to define the R\'edei symbol $[a,b,c]$
for $a, b, c \in\Que^*/{\Que^*}^2$ satisfying
\eqref{eqn:hilbabc} and \eqref{eqn:abccoprime} from
the Introduction, i.e., with relative quadratic Hilbert symbols
$$
(a,b)_p=(a,c)_p=(b,c)_p=1
$$
at all primes $p$, and associated discriminants satisfying the
coprimality condition
$$
\gcd(\Delta(a),\Delta(b),\Delta(c))=1.
$$

\begin{definition}
\label{def:redeisymbol}
For non-trivial $a, b, c \in\Que^*/{\Que^*}^2$ satisfying
\eqref{eqn:hilbabc} and \eqref{eqn:abccoprime},
let $K=\Que(\sqrt{ab})\subset F_{a,b}$ be minimally ramified
over $E=\Que(\sqrt a, \sqrt b)$, as in Definition $\ref{def:minram}$.
Then the R\'edei symbol 
$$
[a,b,c]\in\Gal(F_{a,b}/E)=\FF_2
$$
is defined as
\begin{equation}
\label{eqn:redeisymbol}
[a,b,c]= \Art_c (F_{a,b}/K) =
\begin{cases}
	\Art (\gothc, F_{a,b}/K) &\text{if}\quad c>0;\\
	\Art (\gothc\infty, F_{a,b}/K) &\text{if}\quad c<0.\\
\end{cases}
\end{equation}
Here $\gothc$ is an integral $\Ocal_K$-ideal of norm $|c_0|$,
with $c_0$ the squarefree integer in the class of $c$,
and $\infty$ denotes an infinite prime of $K$.

If one of $a$, $b$, or $c$ is trivial in $\Que^*/{\Que^*}^2$, we
take $[a,b,c]=0$.
\end{definition}
\noindent
With this definition,
which we show in Corollary \ref{cor:indepF}
to be independent of the choice of the minimally ramified extension
$K\subset F_{a,b}$,
the R\'edei symbol becomes perfectly symmetric in its 3 arguments.
\emph{R\'edei's reciprocity law} is the following precise version
of Theorem \ref{thm:reclaw}.
\begin{theorem}
For $a, b, c \in\Que^*/{\Que^*}^2$ satisfying
\eqref{eqn:hilbabc} and \eqref{eqn:abccoprime},
the symbol~\eqref{eqn:redeisymbol} is well-defined,
linear in each of its arguments, and satisfies
$$
[a,b,c]=[b,a,c]=[a,c,b]\in \FF_2.
$$
\end{theorem}
\noindent
We say that the R\'edei symbol $[a,b,c]$ \emph{is defined} if its
arguments $a, b, c \in\Que^*/{\Que^*}^2$ satisfy 
the conditions \eqref{eqn:hilbabc} and \eqref{eqn:abccoprime}.

The inclusion of the infinite prime in the definition \eqref{eqn:redeisymbol}
almost by tautology leads to the following useful property.
\begin{proposition}
\label{prop:ab-ab}
Let $D=d_1d_2$ be a decomposition of the second kind
as in Definition $\ref{def:redeisym}$.
Then the R\'edei symbol $[d_1,d_2,-d_1d_2]\in \FF_2$ is defined and
equals~$0$.
\end{proposition}
\begin{proof}
Multiplying $(d_1,d_2)_p=1$ in \eqref{eqn:redeisymbol} by the 
trivial symbols $(d_1,-d_1)_p$ and $(-d_2,d_2)_p$, we obtain
$(d_1,-d_1d_2)_p=(-d_1d_2,d_2)=1$, so $[d_1,d_2,-d_1d_2]$
satisfies \eqref{eqn:hilbabc}, and obviously also \eqref{eqn:abccoprime}.

Suppose $D=d_1d_2<0$. 
Then the principal ideal $(\sqrt D)$ in the ring of integers
of $K=\Que(\sqrt D)$ of norm $c=-D=-d_1d_2>0$ is trivial in $C(D)$,
so its Artin symbol acts trivially on any unramified abelian extension
$K\subset F$.
In case $D$ is odd, $c$ is squarefree, and 
definition \eqref{eqn:redeisymbol} with $a=d_1$, $b=d_2$, and
$\gothc=(\sqrt D)$ yields the desired equality $[a,b,c]=[d_1,d_2,-d_1d_2]=0$.
For $D<0$ even, $\gothc=(\frac{1}{2}\sqrt D)$ of squarefree norm $-d_1d_2/4$
does the same.

In the case $D=d_1d_2>0$, the class of the ideal
$\gothc=(\sqrt D)$ of norm $D=d_1d_2$ in~$C(D)$ is
the Frobenius at infinity $F_\infty$ from \eqref{eqn:narrow-ordinary},
which acts as $\Art(\infty, F/K)$ on the finitely unramified
abelian extension $K\subset F$.
For $D$ odd and $c=-D<0$ squarefree, $[a,b,c]=[d_1,d_2,-d_1d_2]$
now corresponds to the action of the \emph{square} of 
$\Art(\infty, F/K)$ on $F$, which is \emph{also} the identity,
yielding $[d_1,d_2,-d_1d_2]=0$.
For $D>0$ even, $\gothc=(\frac{1}{2}\sqrt D)$ does the job.
\end{proof}

\section{Proving Redei reciprocity}
\label{S:Proving Redei reciprocity}

\noindent
We already mentioned that R\'edei's original definition
is different from \eqref{eqn:redeisymbol}. 
Not only does he omit a contribution of the infinite prime, putting
$[a,b,-c]=[a,b,c]$, he also requires at least one of
$\Delta(a)$ and $\Delta(b)$ to be odd,
making a symbol like $[-1,2,p]$ in \eqref{eqn:12p} undefined.
The resulting reciprocity law \cite{Redei}*{Satz 4} has
superfluous 2-adic restrictions on the entries, and for $bc<0$
the symbols $[a,b,c]$ and $[a,c,b]$,
which are only both defined for $\Delta(a)$ without prime factors
congruent to $3\mod 4$, differ by a product of four
quadratic and biquadratic symbols.

In his 2007 thesis, Corsman found that including an Artin
symbol at infinity for $c<0$ leads to a perfectly symmetric version of
the reciprocity law.
Both his definition of the symbol and his proof of the law
rely heavily on an incorrect lemma \cite{Corsman}*{Lemma 5.1.2}
claiming that the assumptions \eqref{eqn:hilbabc}
and \eqref{eqn:abccoprime} guarantee
the existence of an extension $K\subset F$ in \eqref{eqn:redeisymbol}
that is unramified at \emph{all} primes $p\nmid\gcd(\Delta(a),\Delta(b))$.
Smith's paper on the average 8-rank behavior of imaginary
quadratic class groups also has an incorrect version of the reciprocity
law \cite{Smith}*{Proposition 2.1}
that disregards the subtleties at both infinite and dyadic primes.

\bigskip\noindent
We now let $a$, $b$, and $c$ be squarefree integers different from 1
satisfying \eqref{eqn:hilbabc} and~\eqref{eqn:abccoprime}.
To see that $[a,b,c]$ is well-defined,
and independent of the many choices that go into the definition
of the symbol,
we first note, using Lemma \ref{lemma:twistedF}, that
an extension $K\subset F_{a,b}$ in \eqref{eqn:redeisymbol}
that is minimally ramified over $E$ does exist,
and that it is unique up to twisting by $t\in T_{a,b}$.

Let $K\subset F$ be minimally ramified over $E=\Que(\sqrt a, \sqrt b)$,
and~$p$ a prime dividing~$c$.
Then~$p$ is split or ramified in $\Que(\sqrt a)$ and in $\Que(\sqrt b)$ by
\eqref{eqn:hilbabc}, and 
unramified in at least one of these fields by \eqref{eqn:abccoprime}.
For a prime $\gothp_K|p$ in $K$, this implies that $\gothp_K$
is of degree~1, and split in the extension $K\subset E$.
Moreover, $\gothp_K$ is \emph{unramified} in $K\subset F$ for primes $p|c$.
Indeed, for odd $p$ we are in case (1) of Definition~\ref{def:minram}
by \eqref{eqn:abccoprime}.
For $2|c$ at least one of $\Delta(a), \Delta(b)$ is odd,
say $\Delta(b)$, and then the condition $(b,c)_2=(\Delta(b), 2)_2=1$
in \eqref{eqn:hilbabc}
shows that we have $\Delta(b)\congr 1\mod 8$, putting us in
case (2) of Definition~\ref{def:minram}.
Thus $\Art(\gothp_K, F/K)\in \Gal(F/E)$ is
a well-defined element of $\Gal(F/\Que)$.
As $\Gal(F/E)$ is contained in the center of $\Gal(F/\Que)$,
and equal to it if $\Que\subset F$ is dihedral,
\begin{equation}
\label{eqn:p-partabc}
[a,b,c]_{F, p} = \Art(\gothp_K, F/K) \in \Gal(F/E)
\end{equation}
only depends on $F$ and $p$, not on $\gothp_K|p$ in $K$.
For $p\nmid c$ we put $[a,b,c]_{F,p}=\id_F$.

For $c<0$, we have $a, b>0$ by condition \eqref{eqn:hilbabc} for $p=\infty$,
so $E=\Que(\sqrt a, \sqrt b)$ is totally real, and the decomposition group at
every infinite prime of $F$ is generated by the Frobenius at infinity
$$
[a,b,c]_{F, \infty}=\Art(\infty, F/K) \in \Gal(F/E).
$$
For $c>0$ we put $[a,b,c]_{F,\infty}=\id_F$.

With this notation, the R\'edei symbol in \eqref{eqn:redeisymbol}
becomes a product
\begin{equation}
\label{eqn:prodp-parts}
[a,b,c]=\prod_{p\le\infty} [a,b,c]_{F, p}\ \in \Gal(F/E)
\end{equation}
of its \emph{$p$-parts}.
The infinite product \eqref{eqn:prodp-parts} is well-defined in
$\Gal(F/E)$, as we can only have $[a,b,c]_{F,p}\ne \id_F$ for primes $p|c$,
with $\infty|c$ having the meaning $c<0$.

As the prime $\gothp_K$ in the Artin symbol 
$\Art(\gothp_K, F/K)=[a,b,c]_{F,p}$ for $p|c$ in \eqref{eqn:p-partabc}
splits in $K\subset E$, we can view it as the Artin symbol of 
a prime $\gothp_E|p$ of $E$ in the quadratic extension 
$E\subset F=E(\sqrt\beta)=E(\sqrt{\beta'})$.
As $\gothp_E$ is unramified in $E\subset F$, its norm to $K_a$
is a prime $\gothp$ of degree 1 over $p$ in~$K_a$ that
is unramified in at least one of the quadratic extensions
$K_a(\sqrt\beta)$ and $K_a(\sqrt{\beta'})$ of $K_a$.
Replacing~$\gothp$ by a conjugate prime in $K_a$ if necessary, 
we can take it to be unramified in $K_a\subset K_a(\sqrt\beta)$.
We can then compute the $p$-part of $[a,b,c]$ as
\begin{equation}
\label{eqn:ppartoverKa}
[a,b,c]_{F,p} = \Art(\gothp, K_a(\sqrt\beta)/K_a)\in\{\pm1\}.
\end{equation}
This shows that $[a,b,c]_{F,p}$ is essentially a Legendre symbol
$\legendre{\beta}{\gothp}$ in the field $K_a$. 
The reason that we choose its value to lie in $\{\pm1\}$
rather than in $\FF_2$, which is of course `only' a matter of notation,
is not just a Legendre symbol tradition,
or the fact that Galois groups like $\Gal(F/E)$ tend to be written as
multiplicative groups.
The point is that, for $\gothp|p$ unramified in
$K_a\subset K_a(\sqrt\beta)$, the $p$-part of $[a,b,c]$
\emph{is} the quadratic Hilbert symbol
\begin{equation}
\label{eqn:ppartishilb}
[a,b,c]_{F,p} = (\beta, \pi)_\gothp\in\{\pm1\}
\end{equation}
of $\beta$ and a uniformizer $\pi$ in the completion of
$K_a$ at $\gothp$.
For $c<0$ and $p=\infty$, we have $[a,b,c]_{F,\infty} = (\beta, -1)_\gothp$,
as the archimedean nature of $F=E(\sqrt\beta)$ is determined 
by the sign of $\beta$ at a real prime $\gothp$ of $K_a$.
Respecting tradition, we have refrained from using Hilbert symbols
with values in $\FF_2$.

It is clear from the symmetry in $a$ and $b$ of the definition of
the R\'edei symbol $[a,b,c]$ that we have $[a,b,c]=[b,a,c]$ whenever
the symbol is defined.
In order to prove the non-trivial \emph{reciprocity law}
$[a,b,c]=[a,c,b]$ in Theorem \ref{thm:reclaw}, we choose a minimally ramified
extension $F=E(\sqrt\beta)$ of $K=\Que(\sqrt{ab})$ as in \eqref{eqn:Fxyz}
in order to express $[a,b,c]$
as a product of $p$-parts $[a,b,c]_{F,p}$ as in \eqref{eqn:prodp-parts},
and similarly a minimally ramified
extension $F'=E'(\sqrt\gamma)$ of $K'=\Que(\sqrt{ac})$ 
in order to express $[a,c,b]$ as a product of $[a,c,b]_{F',p}$.
Here $\beta, \gamma\in \Que(\sqrt a)^*$ are elements of norm 
$b, c\in \Que^*/{\Que^*}^2$, and the fields
$F$ and $F'$ are the normal closures of $\Que(\sqrt a,\sqrt\beta)$
and $\Que(\sqrt a,\sqrt\gamma)$, respectively.
In the spirit of \eqref{eqn:ppartishilb},
we then have the following key lemma.
\begin{lemma}
\label{lemma:keylemma}
Let $a,b,c\in\Que^*/{\Que^*}^2$ be non-trivial elements
satisfying \eqref{eqn:hilbabc} and \eqref{eqn:abccoprime}, and 
$F=E(\sqrt\beta)$ and $F'=E'(\sqrt\gamma)$ 
minimally ramified extensions of $K=\Que(\sqrt{ab})$ and
$K'=\Que(\sqrt{ac})$ defined as above.
For all rational primes $p\le\infty$, we then have
\begin{equation}
\label{eqn:locrec}
[a,b,c]_{F,p}\cdot [a,c,b]_{F',p} =
\prod_{\gothp|p\text{ in }\Que(\sqrt a)} (\beta, \gamma)_\gothp.
\end{equation}
\end{lemma}
\begin{proof}
We denote the left and right hand side of \eqref{eqn:locrec}
by $L_p$ and $R_p$, respectively, and note that $L_p$ and $R_p$ are
symmetric in $b$ and $c$.
Moreover, we can replace $\beta$ (or~$\gamma)$ in~$R_p$ by its conjugate
without changing the value of $R_p$, as the expression $R_p'$ obtained by
replacing $\beta$ by $\beta'$ satisfies
$R_pR_p'=\prod_{\gothp|p}(b, \gamma)_\gothp=(b,c)_p=1$.

For $p=\infty$, condition \eqref{eqn:hilbabc} implies
that at most one of $a, b, c$ is negative.
If they are all positive, we have $L_\infty=1$, and
both $\beta$ and $\gamma$ are totally
positive or negative in the real quadratic field $K_a=\Que(\sqrt a)$.
The symbols $(\beta,\gamma)_\gothp$ at the two
infinite primes of $K_a$ then have the same value,
so we also have $R_\infty=1$.
If only $a$ is negative, we have $L_\infty=1=R_\infty$, 
as the unique infinite prime of $K_a$ is complex.

If $a$ is positive and exactly one of $b$ and $c$, say $c$, is negative,
$L_\infty$ is the Frobenius at infinity in 
$E\subset F=E(\sqrt\beta)$, which equals 1 if 
$\beta\in K_a^*$ is totally positive, and $-1$ if 
$\beta$ is totally negative.
As $\gamma$ has a positive and a negative embedding in~$\RR$,
the same value is taken by the product 
$R_\infty=(\beta, \gamma)_{\infty_1}(\beta, \gamma)_{\infty_2}$
of the Hilbert symbols at the infinite primes of $K_a$.
This settles the case $p=\infty$.

For $p$ a finite prime, take $a, b, c$ to be squarefree integers.
Condition~\eqref{eqn:abccoprime} implies that
$p$ divides at most two of $a, b, c$.
If $p$ divides $b$, it is split or ramified in $K_a$, and
$\beta$ is, up to squares in $K_a^*$, a uniformizer at a prime
$\gothp_1|p$ and, in the split case $(p)=\gothp_1\gothp_2$,
a unit at the other prime $\gothp_2|p$ in $K_a$.
If $p$ does not divide $b$, then the minimal ramification of $K\subset F$
implies that $\beta$ is a $p$-unit, up to squares in $K_a^*$.
For odd $p$ this means that $\sqrt\beta\in F$
generates an extension of $K_a$ that is unramified over $p$.
Analogous statements apply to $c$ and $\gamma$.

Suppose first that $p$ is odd.
If $p$ does not divide $bc$, we have $L_p=1=R_p$,
as the Hilbert symbols $(\beta,\gamma)_\gothp$ at $\gothp|p$
are equal to 1 for $p$-units $\beta$ and $\gamma$.
If $p$ divides exactly one of $b, c$, say~$c$,
we can take $\beta$ to be a $p$-unit, with square root in $F$ that
is unramified over $p$, and $\gamma$ a uniformizer at a prime $\gothp_1|p$.
By \eqref{eqn:ppartishilb}, we then have
\begin{equation}
\label{eqn:artishilb}
L_p = [a,b,c]_{F, p} = (\beta,\gamma)_{\gothp_1}. 
\end{equation}
In the split case $(p)=\gothp_1\gothp_2$, we further have
$(\beta,\gamma)_{\gothp_2}=1$, as
both $\beta$ and $\gamma$ are units at $\gothp_2$.
This yields $L_p=R_p$ both in the ramified and in the split case.

If $p$ divides both $b$ and $c$, it does not divide $a$, so
we are in the split case $(p)=\gothp_1\gothp_2$ in $K_a$.
After replacing $\beta$ by its conjugate, if necessary, 
$\beta$ is a unit at $\gothp_1$ and a uniformizer at $\gothp_2$, 
whereas $\gamma$ is a uniformizer at~$\gothp_1$ and a unit at $\gothp_2$.
Again by \eqref{eqn:ppartishilb},
\begin{equation}
\label{eqn:twofrobs}
L_p = [a,b,c]_{F, p} \cdot [a,c,b]_{F', p} =
	(\beta,\gamma)_{\gothp_1} (\beta,\gamma)_{\gothp_2} = R_p,
\end{equation}
so we have proved our lemma for odd $p$.

For $p=2$, we need a finer distinction as $2\nmid b$, and even
$2\nmid\Delta(b)$, does not imply that the minimally ramified extension
$K\subset F$ is unramified over 2, and that $\sqrt \beta$
generates a subextension of $K_a\subset F$ that is unramified over 2.
For $2\nmid\Delta(b)$, or $b\congr 1\mod 4$, 
Definition \ref{def:minram} shows that it does in all cases
except in the case $a\congr -1(4)$ and $b\congr5\mod8$.
For $b\congr -1\mod 4$, when 2 divides $\Delta(b)$ but not $b$, we do know
that $\beta$ is, up to squares in $K_a^*$, a 2-adic unit.
Moreover, for $\Delta(b)$ even and 2 split in $K_a$,
the extension $K_a\subset K_a(\sqrt\beta)$ is unramified at 
one prime over~2, and ramified at the other.
Same for $c$ and $\gamma$.

Suppose first that $bc$ is odd.
Then we have $L_2=1$, and we take $\beta$ and $\gamma$ 
to be 2-units.
By the condition $(b,c)_2=1$ at least one of $b, c$, say $b$, is $1\mod 4$.
For $c\congr-1\mod4$, the condition $(a,c)_2=1$ implies
$a\not\congr-1\mod4$, so the minimally ramified extension
$K\subset F$ is unramified over 2, and all
Hilbert symbols $(\beta, \gamma)_\gothp$ at primes $\gothp|2$ in $K_a$
occurring in $R_2$ equal 1,
as $\gamma$ is a unit at $\gothp$ and $K_a\subset K_a(\sqrt\beta)$
is unramified at $\gothp$.
For $c\congr 1\mod 4$, Definition \ref{def:minram} tells us
that we are in the same situation, with $R_2=1$ because
one of $\beta$, $\gamma$ is a $\gothp$-unit and the other 
has a $\gothp$-unramified square root, provided that
either we have $a\not\congr-1\mod 4$ or one of $b, c$ is $1\mod 8$.
The remaining special case $a\congr-1\mod 4$ and $b\congr c\congr5\mod 8$
is when both $K_a\subset F$ and $K_a\subset F'$ are ramified at the prime
$\gothp|2$ of $K_a$.
This is where the \emph{minimal} ramification at 2 of the extensions
$K\subset F$ and $K\subset F'$ from Definition~\ref{def:2minram}
is essential:
once more we have $R_2=(\beta,\gamma)_\gothp=1$, as
$\sqrt\beta$ generates an extension of conductor 2 of the
completion $\Que_2(\sqrt a)$ of $K_a$ at $\gothp$, and 
$\gamma$ is 1 modulo $\gothp^2=(2)$ in~$K_a$.
This proves $L_2=1=R_2$ for $bc$ odd.

If exactly one of $b, c$ is even, say $c$, the condition
$(a,c)_2=(b,c)_2=1$ implies $a, b\not\congr5\mod 8$.
For $b\congr1\mod 8$, the minimally ramified extension $K\subset F$,
and therefore $K_a\subset K_a(\sqrt \beta)$, is unramified over 2.
In this case, we have
$$
L_2=[a,b,c]_{F,2}=(\beta,\gamma)_{\gothp_1} =R_2
$$
just as in the case of odd $p$, as we can take
$\gamma$ to be a uniformizer at $\gothp_1|2$ and, in the split case,
a unit at the other prime $\gothp_2$.
In the other case
$b\congr-1\mod4$ both $\Delta(b)$ and $\Delta(c)$ are even, so 
we have $a\congr 1\mod 8$ and $(2)=\gothp_1\gothp_2$ in $K_a$.
In this case,
$\sqrt\beta$ and $\sqrt\gamma$ generate extensions of $K_a$ that
are ramified at one prime over 2, and unramified at the other.
Replacing $\beta$ or $\gamma$ by their conjugate if necessary, we can
assume that $K_a\subset K_a(\sqrt \beta)$ is unramified at $\gothp_1$
and $K_a\subset K_a(\sqrt \gamma)$ unramified at~$\gothp_2$.
Up to squares, $\gamma$ is a then a uniformizer at $\gothp_1$ and
$\beta$ a unit at $\gothp_2$, so we have
$$
L_2=[a,b,c]_{F,2}=(\beta,\gamma)_{\gothp_1}=
(\beta,\gamma)_{\gothp_1}(\beta,\gamma)_{\gothp_2}=R_2.
$$
Finally, for $b$ and $c$ both even, we are also in the split case,
as $\Delta(a)$ is odd and $(a,b)_2=(a,2)_2=1$ implies $a\congr1\mod8$.
As above, we can choose $K_a\subset K_a(\sqrt \beta)$ unramified at $\gothp_1$
and $K_a\subset K_a(\sqrt \gamma)$ unramified at $\gothp_2$.
Up to squares, this makes $\beta$ a uniformizer at $\gothp_2$
and $\gamma$ a uniformizer at $\gothp_1$.
We obtain
$$
L_2 = [a,b,c]_{F, 2} [a,c,b]_{F', 2} =
        (\beta,\gamma)_{\gothp_1} (\beta,\gamma)_{\gothp_2} = R_2,
$$
and we have finished the proof of Lemma \ref{lemma:keylemma}.
\end{proof}
\noindent
{\bf Proof of Theorem \ref{thm:reclaw}.}
By Lemma \ref{lemma:keylemma}, the sum in $\FF_2$ of the R\'edei symbols 
$[a,b,c]$ and $[a,c,b]$, when defined as in \eqref{eqn:redeisymbol}
with the help of $F=E(\sqrt\beta)$ and $F'=E'(\sqrt\gamma)$,
respectively, is the additive analogue of
$\prod_{\gothp\le\infty} (\beta, \gamma)_\gothp\in\{\pm1\}$,
where the product ranges over all primes $\gothp\le\infty$ of $\Que(\sqrt a)$.
By the product formula for Hilbert symbols, this product is equal to 1,
so we have $[a,b,c]=[a,c,b]$, as desired.
As we can trivially swap $a$ and $b$ in $[a,b,c]$, this shows that
the R\'edei symbol is perfectly symmetric in its 3 arguments.

The linearity of $[a,b,c]$ in $c$ is clear from
its description as a product of Artin symbols 
$[a,b,c]_p$ of order 2 at the primes $p|c$. 
It must therefore be linear in all arguments.
\qed
\begin{corollary}
\label{cor:indepF}
The value of the symbol $[a,b,c]$ in \eqref{eqn:redeisymbol}
is the same for all $K\subset F_{a,b}$ that are minimally ramified 
over $\Que(\sqrt a, \sqrt b)$.
\end{corollary}
\begin{proof}
By Theorem \ref{thm:reclaw}, the symbol is equal to $[a,c,b]$, which
is defined independently of a choice $K\subset F_{a,b}$.

One can of course also prove this directly: by Lemma \ref{lemma:twistedF},
two $F$'s that are minimally ramified over $\Que(\sqrt a, \sqrt b)$
differ by a twist $t\in T_{a,b}$, and 
twisting $F=F_{a,b}$ in \eqref{eqn:redeisymbol}
changes the value of $[a,b,c]$ by $\chi_t(c)$, which equals 
0 for $t\in T_{a,b}$ 
by the conditions \eqref{eqn:hilbabc} and \eqref{eqn:abccoprime}.
\end{proof}
\noindent
Even though the symbol $[a,b,c]$ itself is independent of the choice
of $F$ in \eqref{def:redeisymbol}, its $p$-parts $[a,b,c]_{F,p}$
in \eqref{eqn:prodp-parts} {\it do\/} depend on the 
minimally ramified extension $K\subset F$.

It is also possible to define $[a,b,c]$ as an Artin symbol in an 
abelian extension $K\subset \calF_{a,b}$ that is uniquely defined in
terms of $a$ and $b$.
For any minimally ramified extension
$K\subset F$ as in \eqref{eqn:Fxyz}, we can take the compositum 
$$
\calF_{a,b}=FG_{a,b},
$$
of $F$ with the multiquadratic extension $G_{a,b}$ 
obtained by adjoining the square roots $\sqrt t$ of the
elements $t\in T_{a,b}$ from \eqref{eqn:Tab}.
By Lemma \ref{lemma:twistedF}, the number field
$\calF_{a,b}$ is the compositum of \emph{all} minimally ramified
extensions $K\subset F$, so it is uniquely defined in terms of $a$ and $b$.
We now replace $F_{a,b}$ by $\calF_{a,b}$ in \eqref{eqn:redeisymbol} and 
define the R\'edei symbol
$
[a,b,c]\in \Gal(\calF_{a,b}/G_{a,b})=\FF_2
$
as
\begin{equation}
\label{invariantdef}
[a,b,c]=\Art_c(\calF_{a,b}/K)=
\begin{cases}
        \Art (\gothc, \calF_{a,b}/K) &\text{if}\quad c>0;\\
        \Art (\gothc\infty, \calF_{a,b}/K) &\text{if}\quad c<0.\\
\end{cases}
\end{equation}

\noindent
Although Definition \eqref{invariantdef} is in many ways 
the `correct' definition of $[a,b,c]$,
it has the psychological disadvantage of being defined
using a field $\calF_{a,b}$ that is potentially very large.
For the proof of the reciprocity of the symbol, and for actual 
computations of R\'edei symbols, the $p$-parts of $[a,b,c]$,
which are simply Legendre symbols in quadratic fields
such as $K_a=\Que(\sqrt{a})$ by \eqref{eqn:ppartoverKa},
are handled more easily.

\section{Governing fields}
\label{S:Governing fields}

\noindent
An immediate application of R\'edei's reciprocity law in the form
we have stated it is the existence of \emph{governing fields}
for the 8-rank of the narrow class group $C(dp)$ of the quadratic field
$\Que(\sqrt{dp})$, with $d$ a fixed squarefree integer and $p$
a variable prime.
By this, we mean that there exists a normal number field
$\Omega_{8,d}$ with the property that for primes $p, p'\nmid d$ that are
coprime to its discriminant and have the same Frobenius conjugacy class
in $\Gal(\Omega_{8,d}/\Que)$, the groups $C(dp)/C(dp)^8$ and 
$C(dp')/C(dp')^8$ are isomorphic.

Theorem \ref{theorem:2-rank} trivially implies that
$\Omega_{2,d}=\Que(i)$ is a governing field for the 2-rank
of $C(dp)$. 
By the explicit form \eqref{eqn:R4entries} of
Theorem \ref{theorem:4-rank}, we can take the multi-quadratic field
$$
\Omega_{4,d}=\Que(i, \{\sqrt p: p|d\ \hbox{prime}\})
$$
as a governing field for the 4-rank of $C(dp)$.

Now suppose $p$ and $p'$ are primes that are unramified in $\Omega_{4,d}$
and have the same Artin symbol in $\Gal(\Omega_{4,d}/\Que)$.
Then the R\'edei matrices $R_4$ and $R'_4$ for $C(dp)$ and $C(dp')$
as given in \eqref{eqn:R4entries} \emph{coincide} if
the primes in $dp$ and $dp'$ are numbered in the obvious compatible way.
This implies that the 8-rank maps in \eqref{eqn:R8} can be described by
matrices $R_8$ and $R'_8$ for $C(dp)$ and $C(dp')$ with entries
given by \eqref{eqn:R8entries} that may be compared `entry-wise'.
In other words, every entry $[d_1, d_2, m]$ from
Definition \ref{def:redeisym} in the matrix $R_8$ for $D=d_1d_2\in\{dp, 4dp\}$
corresponds to a R\'edei symbol $[d'_1, d'_2, m']$ for $R_8'$ 
in which the arguments are obtained by replacing every prime factor $p$
in the entries of $[d_1, d_2, m]$ by the factor $p'$.

Possibly switching the role of $d_1$ and $d_2$, we may assume that
all symbols $[d_1, d_2, m]$ in $R_8$ have $p\nmid d_1$.
Moreover, if we have $p|m$ in a symbol $[d_1, d_2, m]$, we can
add the trivial symbol $[d_1, d_2, -d_1d_2]$ from
Proposition \ref{prop:ab-ab} to it to rewrite it as a symbol
\begin{equation}
[d_1, d_2, m]=[d_1, d_2, -d_1d_2/m]=[d_1, d_2, dp/m]
\end{equation}
with $p\nmid (dp/m)$.
Thus, we may write the entries of $R_8$ as $[d_1, d_2, m]$ with $p\nmid d_1m$.
Then the entries of $R_8'$ become $[d'_1, d'_2, m']=[d_1, d'_2, m]$.

In order to show that the value of $[d_1, d_2, m]$ for $p\nmid d_1m$
is governed by the splitting behavior of $p$ in some
finite extension of $\Omega_{4,d}$, it suffices to rewrite it 
using Theorem \ref{thm:reclaw} as
$$
[d_1, d_2, m]=[d_1, m, d_2],
$$
and observe that we now have
$[d_1, m, d_2]=[d_1, m, d'_2]$ for $d'_2=p'd_2/p$ whenever 
$p$ and $p'$ have the same splitting behavior in $\Omega_{4,d}(\sqrt\mu)$,
with $\mu\in \Que(\sqrt{d_1})$ an element with norm in $m\cdot{\Que^*}^2$
that generates a minimally ramified extension of $K=\Que(\sqrt{md_1})$
as in \eqref{eqn:Fxyz} for $(a, b)=(d_1, m)$.
Note that $\Omega_{4,d}\subset \Omega_{4,d}(\sqrt\mu)$ is
unramified outside $2d$.

Taking $\Omega_{8,d}$ to be the compositum of the fields
$\Omega_{4,d}(\sqrt\mu)$ arising for each of the entries $[d_1,d_2,m]$ of $R_8$,
we see that $R_8$ and $R_8'$ coincide for primes $p$, $p'$ having the 
same Frobenius conjugacy class in $\Gal(\Omega_{8,d}/\Que)$.
We arrive at the following theorem, which was proved
in a more involved way in 1988 in \cite{Stev}.
The short proof we gave above already occurs in \cite{Corsman}.
\begin{theorem}
A governing field $\Omega_{8,d}$ for the 8-rank of $C(dp)$ exists, and 
one can take for it the maximal exponent $2$ extension of $\Omega_{4,d}$
unramified outside $2d$.
\qed
\end{theorem}
\noindent
The existence of $\Omega_{8,d}$ implies, by the Chebotarev density theorem,
that we can compute the density of the set of primes $p$ for which 
$C(dp)$ has prescribed 2-, 4- and 8-rank.
Cohn and Lagarias \cite{CohnLag} conjectured in 1983 that such
governing fields $\Omega_{2^k}$ should exist for the $2^k$-rank of $C(dp)$
for all $k\ge1$.
Recent work of Milovic and Koymans \cites{Milovic, MilKoymans} establishes
density results for 16-ranks of class groups $C(dp)$
with cyclic 2-part, such as $C(-2p)$,
with error terms that are ``too good" to come
from a governing field, making it unlikely that the conjecture
holds for $2^k$-ranks with $k\ge4$.
Smith \cite{Smith2} nevertheless arrives at proving
\emph{average} distributions for higher $2^k$-ranks by `governing less',
focusing not on individual quadratic fields $K$ but on their
related behavior in well-chosen families.

\section{The negative Pell equation}
\label{S:negative Pell}

\noindent
By \eqref{eqn:Hreal},
proving Conjecture \ref{conj:negpell} entails
controlling the archimedean character of the full narrow
Hilbert class field $H$ of $K=\Que(\sqrt D)$ for the discriminants $D$
in the thin set $\calD$ of discriminants for which the genus field
$H_2$ is totally real.
The 4-Hilbert class field $H_4$ of $K$ can be given explicitly, in the
sense of Corollary \ref{cor:F4unr}, as a
compositum of cyclic quartic extensions $K\subset F_{d_1,d_2}$
that are unramified over the fields $E=\Que(\sqrt {d_1},\sqrt{d_2})$, with
$D=d_1d_2$ ranging over (a basis of) the decompositions of $D$
of the second kind, as characterized in Lemma \ref{lemma:F2-criterion}.
By Definition \ref{def:redeisymbol},
the archimedean character of the dihedral field $F_{d_1,d_2}$ is given by
the R\'edei symbol $[d_1, d_2, -1]=\Art(\infty, F_{d_1,d_2}/K)$,
which is defined since $d_1, d_2\in\calD$ satisfy $(d_1,-1)_p=(d_2,-1)_p=1$
for all $p$.
\begin{theorem}
\label{thm:d1d2-1}
Let $D=d_1d_2$ be a decomposition of the second kind for $D\in\calD$, and
$K\subset F_{d_1,d_2}$ a corresponding unramified extension as
in $(1)$ of Lemma $\ref{lemma:F2-criterion}$.
Then $F_{d_1,d_2}$ is totally real if and only if we have
$$
[d_1, d_2, -1]=\Legendre{d_1}{d_2}_4+\Legendre{d_2}{d_1}_4=0\in\FF_2.
$$
\end{theorem}
\begin{proof}
Add to $[d_1, d_2, -1]$ the symbol
$[d_1, d_2, -d_1d_2]=0$ from Proposition~\ref{prop:ab-ab},
and use the linearity and reciprocity properties of the R\'edei symbol
together with the special case in Definition \ref{def:aac} to obtain 
the desired equality
\begin{align*}
[d_1, d_2, -1] & = [d_1, d_2, d_1d_2]= [d_1, d_2, d_1]+[d_1, d_2, d_2] \\
               & = [d_1, d_1, d_2]+[d_2, d_2, d_1] = 
\Legendre{d_2}{d_1}_4+\Legendre{d_1}{d_2}_4.
\lqedhere
\end{align*}
\end{proof}

\noindent
Theorem \ref{thm:d1d2-1} was already known to R\'edei, but 
his symbol `without infinite primes' cannot be used to give
the short proof above.
A long proof can be found in \cite{FouvKl}*{pp.\ 2061--2064}.

Theorem \ref{thm:d1d2-1} implies that
$H_4$ is totally real if and only if
all R\'edei symbols $[d_1, d_2, -1]$ 
corresponding to $r_4$ decompositions $D=d_1d_2$ of the second
kind spanning $\widehat C[2]\cap 2\widehat C$ vanish.
The explicit form of the symbol as a sum of two biquadratic symbols 
was used by Fouvry and Kl\"uners \cite{FouvKl} to show that
these $r_4$ different R\'edei symbols vanish
for the expected fraction $2^{-e}$ of all discriminants
in the subset $\calD(e)\subset\calD$ of $D$ having $r_4=e$.
As $\calD(e)$ has density $P(e)$ in $\calD$, a short calculation
\cite{StevPell}*{Corollary 4.4}
involving the explicit values following equation \eqref{eqn:densalpha}
shows that the density of $D\in \calD$ with $H_4$ totally real equals
$$
\sum_{e=0}^\infty 2^{-e} P(e)= 
\sum_{e=0}^\infty 2^{-e} \cdot\alpha\cdot \prod_{j=1}^e (2^j-1)^{-1}=
\frac{2}{3}.
$$
In the direction of Conjecture \ref{conj:negpell},
this yields an upper density $\overline P\le \frac{2}{3}\approx .667$ for
the subset $\calD^-$ of $\calD$. 

For discriminants $D\in \calD$ having $r_8=0$, the necessary condition
for $D\in\calD^-$ that $H_4$ be totally real is also
sufficient, and this gives rise to lower bounds.
Discriminants in $\calD(0)$ trivially have $r_8=r_4=0$, and
the inclusion $\calD(0)\subset \calD^-$ 
yields the lower bound $P(0)=\alpha\approx.419$
from \eqref{eqn:densalpha} for the lower density
$\underline P$ of $\calD^-$ in $\calD$ that we already mentioned 
at the end of Section \ref{section:4-rank}.

Discriminants $D\in\calD(1)$ have a unique non-trivial decomposition 
$D=d_1d_2$ of the second kind, and the kernel of the symmetric matrix
$R_4$ is spanned by the discriminantal divisors $d_1$ and $d_2$. 
In this case R\'edei reciprocity for the special symbols in Definition 
\ref{def:aac} shows that $R_8$ has a matrix representation
$$
R_8=
\begin{pmatrix}
[d_1, d_2, d_1]&[d_1, d_2, d_2]
\end{pmatrix}
=
\left(
\Legendre{d_2}{d_1}_4\quad\Legendre{d_1}{d_2}_4
\right)
$$
in terms of biquadratic symbols.
In view of Theorems \ref{theorem:8-rank} and \ref{thm:d1d2-1},
having $r_8=0$ and $H_4$ totally real now amounts~to 
\begin{equation}
\label{eqn:d1221_4}
\Legendre{d_2}{d_1}_4=\Legendre{d_1}{d_2}_4=1\in\FF_2,
\end{equation}
and, again, 
for the expected fraction $\frac{1}{4}$ of discriminants in $\calD(1)$,
both biquadratic symbols are non-trivial \cite{FouvKl2}.
This yields $P(0)+\frac{1}{4}P(1)=\frac{5}{4}\alpha\approx.524$
as a lower bound for the lower density $\underline P$.

For discriminants $D\in\calD(e)$ with $e\ge2$, the R\'edei symbols in the
matrix $R_8$ are not restricted to biquadratic symbols, but 
Chan, Koymans, Milovic and Pagano \cite{CKMP} show that they can still be
`governed' by an adaptation of the methods
in Smith's recent work \cites{Smith, Smith2}.
They prove that the density of discriminants $D\in \calD$ 
of 4-rank $r_4=e$ for which $r_8=0$ and $H_4$ is totally real
equals $2^{-e(e+3)/2}\cdot\alpha$,
extending the trivial case $e=0$ and the special case $e=1$ above.
This improves the lower bound $\alpha+\frac{1}{4}\alpha$ above to
$\sum_{e=0}^\infty 2^{-e(e+3)/2}\cdot\alpha$.
Thus, the published state of affairs towards Conjecture \ref{conj:negpell},
which claims $\underline P=\overline P=1-\alpha=.581$, becomes
$$
\textstyle
.538\approx \sum_{e=0}^\infty 2^{-e(e+3)/2}\cdot\alpha
\  \le \ 
\underline P 
\  \le \ 
\overline P 
\  \le \ 
2/3\approx .667   .
$$
Improving these bounds involves dealing with the remaining 12.8\% of
discriminants $D\in\calD$ having $r_8>0$.
As this article goes to press (March 2021), Koymans and 
Pagano have announced that they have been able to extend
Smith's techniques \cite{Smith2} to also control these discriminants,
and to prove my full Conjecture \ref{conj:negpell} after almost 30 years.


\end{document}